\documentclass[11pt]{amsart}
\usepackage{graphicx}
\usepackage{amscd}
\usepackage{amsmath}
\usepackage{amsxtra}
\usepackage{amsfonts}
\usepackage{amssymb}
\usepackage{bm}
\usepackage{setspace}
\usepackage{tikz, xcolor}
\allowdisplaybreaks 

\newtheorem{theorem}{Theorem}[section]
\newtheorem{corollary}[theorem]{Corollary}
\newtheorem{lemma}[theorem]{Lemma}
\newtheorem{proposition}[theorem]{Proposition}

\theoremstyle{definition}
\newtheorem{definition}[theorem]{Definition}
\newtheorem{remark}[theorem]{Remark}

\newtheorem{example}[theorem]{Example}
\theoremstyle{remark}

\renewcommand{\theclaim}{\textup{\theclaim}}

\newtheorem*{acknowledgements}{Acknowledgements}

\numberwithin{equation}{section}

\def\openone

{\mathchoice

{\hbox{\upshape \small1\kern-3.3pt\normalsize1}}

{\hbox{\upshape \small1\kern-3.3pt\normalsize1}}

{\hbox{\upshape \tiny1\kern-2.3pt\SMALL1}}

{\hbox{\upshape \Tiny1\kern-2pt\tiny1}}}

\makeatletter

\newbox\ipbox

\newcommand{\ip}[2]{\left\langle #1\, , \,#2\right\rangle}
\newcommand{\diracb}[1]{\left\langle #1\mathrel{\mathchoice

{\setbox\ipbox=\hbox{$\displaystyle \left\langle\mathstrut
#1\right.$}

\vrule height\ht\ipbox width0.25pt depth\dp\ipbox}

{\setbox\ipbox=\hbox{$\textstyle \left\langle\mathstrut
#1\right.$}

\vrule height\ht\ipbox width0.25pt depth\dp\ipbox}

{\setbox\ipbox=\hbox{$\scriptstyle \left\langle\mathstrut
#1\right.$}

\vrule height\ht\ipbox width0.25pt depth\dp\ipbox}

{\setbox\ipbox=\hbox{$\scriptscriptstyle \left\langle\mathstrut
#1\right.$}

\vrule height\ht\ipbox width0.25pt depth\dp\ipbox}

}\right. }

\newcommand{\dirack}[1]{\left. \mathrel{\mathchoice

{\setbox\ipbox=\hbox{$\displaystyle \left.\mathstrut
#1\right\rangle$}

\vrule height\ht\ipbox width0.25pt depth\dp\ipbox}

{\setbox\ipbox=\hbox{$\textstyle \left.\mathstrut
#1\right\rangle$}

\vrule height\ht\ipbox width0.25pt depth\dp\ipbox}

{\setbox\ipbox=\hbox{$\scriptstyle \left.\mathstrut
#1\right\rangle$}

\vrule height\ht\ipbox width0.25pt depth\dp\ipbox}

{\setbox\ipbox=\hbox{$\scriptscriptstyle \left.\mathstrut
#1\right\rangle$}

\vrule height\ht\ipbox width0.25pt depth\dp\ipbox}

} #1\right\rangle}

\newcommand{\beq}{\begin{equation}}

\newcommand{\eeq}{\end{equation}}

\newcommand{\nul}{\textbf{null}}

\newcommand{\cj}[1]{\overline{#1}}

\newcommand{\bz}{\mathbb{Z}}

\newcommand{\br}{\mathbb{R}}
\newcommand{\bc}{\mathbb{C}}

\newcommand{\bn}{\mathbb{N}}

\newcommand{\arr}[1]{\stackrel{#1}{\rightarrow}}

\def\blfootnote{\xdef\@thefnmark{}\@footnotetext}

\renewcommand{\mod}{\operatorname{mod}}

\input xy
\xyoption{all}
\usepackage{amssymb}

\newcommand{\Span}{\overline{\operatorname*{span}}}

\def\N{\mathbb{N}}

\def\-{^{-1}}

\def\ty{\emptyset}

\def\C{\mathbb{C}}
\def\Z{\mathbb{Z}}

\begin{document}

\title[Cuntz algebras and random walks]{Representations of Cuntz algebras associated to random walks on graphs}
\author{Dorin Dutkay}
\address{[Dorin Ervin Dutkay] University of Central Florida\\
	Department of Mathematics\\
	4000 Central Florida Blvd.\\
	P.O. Box 161364\\
	Orlando, FL 32816-1364\\
U.S.A.\\} \email{Dorin.Dutkay@ucf.edu}

\author{Nicholas Christoffersen}
\address{[Nicholas Christoffersen] Louisiana State University\\
	Department of Mathematics\\
    Baton Rouge, LA 70803-4918\\
U.S.A.\\} \email{nchri15@lsu.edu}
\thanks{}

\begin{abstract}
Motivated by the harmonic analysis of self-affine measures, we introduce a class of representations of the Cuntz algebra associated to random walks on graphs. The representations are constructed using the dilation theory of row coisometries. We study these representations, their commutant and the intertwining operators. 
\end{abstract}
\maketitle \tableofcontents
\section{Introduction}

\begin{definition}\label{de1.1}
Let $\Lambda$ be a finite set of cardinality $|\Lambda|=N$, $N\in\bn$, $N\geq 2$. A representation of the {\it Cuntz algebra} $\mathcal O_N$, is a family of $N$ isometries $(S_{\lambda})_{\lambda\in\Lambda}$ on a Hilbert space $\mathcal H$, such that the isometries have mutually orthogonal ranges whose sum is the entire space $\mathcal H$. These properties are expressed in the Cuntz relations 
\begin{equation}
S_{\lambda}^*S_{\lambda'}=\delta_{\lambda\lambda'}I_{\mathcal H}, \quad(\lambda,\lambda'\in\Lambda),\quad \sum_{\lambda\in\Lambda}S_{\lambda}S_{\lambda}^*=I_{\mathcal H}.
\label{eqde1.1.1}
\end{equation}
\end{definition}

The Cuntz algebra was introduced in \cite{Cun77} as a simple, purely infinite $C^*$-algebra. The representation theory of Cuntz algebras is extremely rich, unclassifiable even \cite{Gli60,Gli61}. 
The representations of Cuntz algebras have been proved to have applications in mathematical physics \cite{BrJo02, Bur04, GN07, AK08, Kaw03, Kaw06, Kaw09, KHL09,MR2821778,MR3024468,MR3428850,MR4065254}, in wavelets \cite{MR1743534,MR2063755,MR2277210,DuJo08a, DuJo07a, Jor06, Jor01}, harmonic analysis \cite{ DuJo9a, DuJo07b,DuJo12}, and in fractal geometry \cite{DuJo06a, DuJo11a}. 

We will introduce here a class of representations of the Cuntz algebra, associated to a random walk. The construction will follow two major steps: from a random walk, one constructs a row coisometry; then, from the row coisometry, using dilation theory, one obtains a representation of the Cuntz algebra.

\begin{definition}\label{de1.2}
A family of operators $(V_{\lambda})_{\lambda\in\Lambda}$ on a Hilbert space $\mathcal K$ is called a {\it row coisometry}, if 
\begin{equation}
\sum_{\lambda\in\Lambda}V_{\lambda}V_{\lambda}^*=I_{\mathcal K}.
\label{eqde1.2.1}
\end{equation}
\end{definition}

It is known \cite{Pop89,BJK00} that every row coisometry can be dilated to a representation of the Cuntz algebra. More precisely

\begin{theorem}\label{th1.1}\cite[Theorem 5.1]{BJK00}
Let $\mathcal K$ be a Hilbert space and $(V_{\lambda})_{\lambda\in\Lambda}$ be a row coisometry on $\mathcal K$. Then $\mathcal K$ can be embedded into a larger Hilbert space $\mathcal H=\mathcal H_V$ carrying a representation $(S_{\lambda})_{\lambda\in\Lambda}$ of the Cuntz algebra $\mathcal O_N$ such that, if $P:\mathcal H\rightarrow\mathcal K$ is the projection onto $\mathcal K$, we have
\begin{equation}
V_{\lambda}^*=S_{\lambda}^*P,
\label{eqth1.1.1}
\end{equation}
(i.e., $S_{\lambda}^*\mathcal K\subset \mathcal K$ and $S_{\lambda}^*P=PS_{\lambda}^*P=V_{\lambda}^*$) and $\mathcal K$ is cyclic for the representation. 

The system $(\mathcal H, (S_{\lambda})_{\lambda\in\Lambda},P)$ is unique up to unitary equivalence, and if $\bm\sigma:\mathcal B(\mathcal K)\rightarrow\mathcal B(\mathcal K)$ is defined by 
\begin{equation}
\bm\sigma(A)=\sum_{\lambda\in\Lambda}V_{\lambda}AV_{\lambda}^*,
\label{eqth1.1.2}
\end{equation}
then the commutant of the representation $\{S_\lambda\}_{\lambda\in\Lambda}'$ is isometrically order isomorphic to the fixed point set $\mathcal B(\mathcal K)^{\bm \sigma}=\{A\in\mathcal B(\mathcal K) : \bm\sigma(A)=A\}$, by the map $A'\mapsto PA'P$. More generally, if $(W_{\lambda})_{\lambda\in\Lambda}$ is another row coisometry on the space $\mathcal K'$, and $(T_{\lambda})_{\lambda\in\Lambda}$ is the corresponding Cuntz dilation, then there exists an isometric linear isomorphism between the intertwiners $U:\mathcal H_V\rightarrow\mathcal H_W$, i.e., operators satisfying 
\begin{equation}
US_{\lambda}=T_{\lambda}U,
\label{eqth1.1.3}
\end{equation}
and operators $V\in\mathcal B(\mathcal K,\mathcal K')$ such that 
\begin{equation}
\sum_{\lambda\in\Lambda}W_{\lambda}VV_{\lambda}^*=V,
\label{eq:th1.1.4}
\end{equation} 
given by the map $U\mapsto V=P_{\mathcal K'}UP_{\mathcal K}$. 
\end{theorem}

\begin{definition}\label{def1.2}
    Given a row coisometry $(V_{\lambda})_{\lambda\in\Lambda}$ on the Hilbert space $\mathcal K$, we call the representation $(S_{\lambda})_{\lambda\in\Lambda}$ of the Cuntz algebra $\mathcal{O}_{\left| \Lambda \right|}$ in Theorem \ref{th1.1}, {\it the Cuntz dilation of the row coisometry} $(V_{\lambda})_{\lambda\in\Lambda}$.
\end{definition}

\begin{definition}\label{def1.3}
Let $\mathcal G=(\mathcal V,\mathcal E)$ be a directed graph, where the set of vertices $\mathcal V$ is finite or countable. For each edge $e\in\mathcal E$ we assume we have a label $\lambda(e)$ chosen from a finite set of labels $\Lambda$, $|\Lambda|=N$. 

We assume in addition, that given a vertex $i$, two different edges $e_1\neq e_2$ from $i$ have different labels $\lambda(e_1)\neq \lambda(e_2)$ and their end vertices are different. We use the notation $i\stackrel{\lambda}{\rightarrow}j$ to indicate that there is an edge from $i$ to $j$ with label $\lambda$, and, in this case we also write $i\cdot \lambda=j$. Thus, for each vertex $i$, there is at most one edge leaving $i$ with label $\lambda$. 

We also assume that for each vertex $i$ there is at most one edge coming into the vertex $i$ with label $\lambda$. 

For each vertex $i$ and each label $\lambda$ we assume that we have an associated complex number $\alpha_{i,\lambda}$, $\alpha_{i,\lambda}=0$ in case there is no edge from $i$ with label $\lambda$, and we assume that 
\begin{equation}
\sum_{\lambda\in\Lambda}|\alpha_{i,\lambda}|^2=1,\quad (i\in\mathcal V).
\label{eqde1.3.1}
\end{equation}

Thus, we have a random walk on the graph $\mathcal G$, the probability of transition from $i$ to $i\cdot\lambda$ with label $\lambda$ being $|\alpha_{i,\lambda}|^2$.

Here is the way we define the representation of the Cuntz algebra $\mathcal O_N$. 

First, we define the row coisometry $(V_\lambda)_{\lambda\in \Lambda}$ on $\mathcal K:=\ell^{2}[\mathcal V]$, 
\begin{equation}
V_\lambda^*(\vec i)=\left\{\begin{array}{cc}
	\alpha_{i,\lambda}\vec j,&\mbox{ if }i\stackrel{\lambda}{\rightarrow}j\\
	0,&\mbox{ otherwise.}
\end{array}\right.
\label{eqde1.3.2}
\end{equation}
(Here, the vertices $i$ in $\mathcal{V}$ are considered as the canonical basis vectors $\vec i:=\delta_i$ for $l^2[\mathcal V]$).

Then, we use Theorem \ref{th1.1}, to construct the Cuntz dilation $(S_\lambda)_{\lambda\in\Lambda}$. We call it {\it the representation of the Cuntz algebra associated to the random walk}.
\end{definition}

The paper is structured as follows: in section 2, we present some motivation from the harmonic analysis of self-affine measures and we show how our representations of the Cuntz algebra appear in that context. In section 3, we present some general properties of the Cuntz representations obtained from dilations of row coisometries. Of particular interest, it is the following fact: the commutant of the Cuntz dilation is an algebra, it has a product structure, by composition. The isometric correspondence between the commutant and the space $\mathcal B(\mathcal K)^{\bm \sigma}$, implies that we get a product structure also on this space. We make this product structure explicit in Proposition \ref{pr2.4}.

Since the general dilation theorem is a bit abstract, in section 4 we present an explicit construction of the Cuntz dilation associated to a random walk on a graph. In section 5, we study the intertwining operators between the Cuntz dilations associated to two random walks on {\it finite} graphs. Of course, in particular, this covers the case of the commutant of the Cuntz dilation. It turns out that these operators are completely determined by the {\it balanced minimal invariant sets}. Here are the definitions. 

The Cuntz representations studied in this paper fit into the more general class of isometric dilations of non-commuting finite rank $N$-tuples, investigated by Davidson, Kribs and Shpigel in \cite{DKS01}, where the authors describe the structure of the free semigroup algebra generated by the dilation and the unitary and similarity invariants of the associated representations. Here we will describe these invariants directly in terms of the properties of the random walk.

\begin{definition}\label{def1.1}
Suppose we have two directed finite graphs $(\mathcal V,\mathcal E)$ and $(\mathcal V',\mathcal E')$, with labels from the same set $\Lambda$, and with weights $(\alpha_{i,\lambda})_{i\in\mathcal V,\lambda\in\Lambda}$ and $(\alpha_{i',\lambda}')_{i'\in\mathcal V',\lambda\in\Lambda}$ respectively.

Given two pairs of vertices $(i,i'),(j,j')\in \mathcal V\times\mathcal V'$, we say that {\it the transition from $(i,i')$ to $(j,j')$ is possible, with word $\lambda=\lambda_1\dots\lambda_n$}, and we write $(i,i')\arr{\lambda} (j,j')$, if there are pairs of vertices $(i,i')=(i_0,i_0'),(i_1,i_1'),\dots,(i_n,i_n')=(j,j')$ such that the transitions $(i_{k-1},i_{k-1}')\arr{\lambda_k}(i_k,i_k')$ is possible, meaning $\alpha_{i_{k-1},\lambda_k}\neq 0$, $\alpha_{i_{k-1}',\lambda_k}'\neq 0$, for all $k=1,\dots,n$. In other words, the transition from $(i,i')$ to $(j,j')$ is possible exactly means that there are paths from $i$ to $i'$ and $j$ to $j'$ with non-zero weights, which have the same length and label. We use also the notations 
$$i\cdot (\lambda_1\dots\lambda_n)=(\dots((i\cdot\lambda_1)\cdot\lambda_2)\cdot\dots)\cdot\lambda_n,$$
$$\alpha_{i,\lambda}=\alpha_{i,\lambda_1}\alpha_{i\cdot\lambda_1,\lambda_2}\dots\alpha_{i\cdot\lambda_1\lambda_2\dots\lambda_{n-1},\lambda_n},$$
and $i\arr\lambda j$ if $j=i\cdot \lambda$ with $\alpha_{i,\lambda}\neq 0$. 

So, the transition $(i,i')\arr{\lambda}(j,j')$ is possible, with $\lambda=\lambda_1\dots\lambda_n$ if $i\arr\lambda j$, $i'\arr\lambda j'$, $\alpha_{i,\lambda}\neq 0$, and $\alpha_{i',\lambda}'\neq 0$.

A non-empty subset $\mathcal M$ of $\mathcal V\times\mathcal V'$ is called {\it invariant}, if for any $(i,i')\in\mathcal M$, and $\lambda=\lambda_1\dots\lambda_n$, if the transition $(i,i')\arr\lambda(j,j')$ is possible, then $(j,j')\in\mathcal M$. The invariant set $\mathcal M$ is called {\it minimal} if it has no proper invariant subsets.

The {\it orbit} $\mathcal O(i,i')$ of a pair of vertices $(i,i')$ is the smallest invariant set that contains $(i,i')$. 
\end{definition}

\begin{definition}\label{def4.9}
For each minimal invariant set $\mathcal M$, pick a point $(i_{\mathcal M},i'_{\mathcal M})$ in $\mathcal M$. For $(i,i')\in\mathcal V\times \mathcal V'$, define the set $F(i,i')$ of paths/words that arrive at one of the points $(i_{\mathcal M},i'_{\mathcal M})$ for the first time; that is $\lambda=\lambda_1\dots\lambda_n$, $n\geq 1$, is in $F(i,i')$ if and only if $(i\cdot\lambda,i'\cdot\lambda)=(i_{\mathcal M},i'_{\mathcal M})$ for some minimal invariant set $\mathcal M$, and $(i\cdot \lambda_1\dots\lambda_k,i'\cdot\lambda_1\dots\lambda_k)\neq (i_{\mathcal N},i'_{\mathcal N})$ for all $1\leq k<n$ and all minimal invariant sets $\mathcal N$.

We say that a minimal invariant set $\mathcal M$ is {\it balanced} if 
\begin{enumerate}
	\item[(i)] For all $(i,i')\in \mathcal M$, and for all $\lambda\in\Lambda$, $|\alpha_{i,\lambda}|=|\alpha_{i',\lambda}'|$.
	\item[(ii)] For all $(i,i')\in \mathcal M$, and all loops $\lambda=\lambda_1\dots\lambda_n$ at $(i,i')$, i.e., $(i,i')\arr\lambda (i,i')$, one has $\alpha_{i,\lambda}=\alpha_{i',\lambda}'$.
\end{enumerate}

For reasons that will be apparent in section 4, we also use the notation $(i,\ty):=\vec i=\delta_i$, for the orthonormal basis of the space $\mathcal K$.
\end{definition}

According to Theorem \ref{th1.1}, the operators that intertwine the two Cuntz dilations associated to the two random walks are in bijective correspondence with the operators in the space $B(\mathcal K,\mathcal K')^{\bm\sigma}$ which is the space of operators $T:\mathcal K\rightarrow\mathcal K'$ with 
$$
T=\sum_{\lambda\in\Lambda}V_\lambda'TV_\lambda^*=:\bm \sigma(T).
$$
The main result of section 5 describes the space $B(\mathcal K,\mathcal K')^{\bm\sigma}$ in terms of balanced minimal invariant sets, thus describing also the intertwining operators between the two Cuntz dilations.

\begin{theorem}\label{th4.12}
 Let 
$$\mathcal C=\left\{ (i_{\mathcal M},i_{\mathcal M}') : \mathcal M\mbox{ balanced minimal invariant set}\right\},$$
and let $\bc[\mathcal C]$ be the space of complex valued functions on $\mathcal C$.

Let $\mathcal B(\mathcal K,\mathcal K')^{\bm\sigma}$ be the space of operators $T:\mathcal K\rightarrow\mathcal K'$ with 
\begin{equation}
T=\sum_{\lambda\in\Lambda}V_\lambda'TV_\lambda^*=:\bm \sigma(T).
\label{eq4.12.1}
\end{equation}
Then there is a linear isomorphism between $\bc[\mathcal C]$ and $\mathcal B(\mathcal K,\mathcal K')^{\bm \sigma}$, $\bc[\mathcal C]\ni c\mapsto T\in\mathcal B(\mathcal K,\mathcal K')^{\bm\sigma}$, defined by 
\begin{equation}
\ip{T(i,\ty)}{(i',\ty)}=\sum_{\lambda\in F(i,i')}\alpha_{i,\lambda}\cj \alpha'_{i',\lambda} c(i\cdot\lambda,i'\cdot \lambda).
\label{eq4.12.2}
\end{equation}
The inverse of this map is $\mathcal B(\mathcal K,\mathcal K')^{\bm\sigma}\ni T\mapsto c\in \bc[\mathcal C]$, 
\begin{equation}
c(i_{\mathcal M},i_{\mathcal M}')=\ip{T(i_{\mathcal M},\ty)}{(i'_{\mathcal M},\ty)},\quad (\mathcal M\in\mathcal C).
\label{eq4.12.3}
\end{equation}
\end{theorem}

We end the paper with some examples, in section 6. 

\section{A motivation from harmonic analysis}
The study of orthogonal Fourier series on fractal measures began with the paper \cite{JP98}, where Jorgensen and Pedersen proved that, for the Cantor measure $\mu_4$ with scale 4 and digits 0 and 2, the set of exponential functions 
$$\left\{e^{2\pi  i\lambda x} : \lambda=\sum_{k=0}^n 4^kl_k, n\in\bn, l_k\in\{0,1\}\right\},$$
is an orthonormal basis in $L^2(\mu_4)$. Such a measure, which possesses an orthonormal Fourier basis of exponential functions is called a {\it spectral measure}. Many more examples of spectral measures have been constructed since, see, e.g., \cite{Str00, DHL19}. For the classical Middle Third Cantor measure, Jorgensen and Pedersen proved that this construction is not possible. Strichartz \cite{Str00} posed the question, if this measure, has a frame of exponential functions. The question is still open at the time of writing this article. 
 In \cite{PiWe17}, Picioroaga and Weber, trying to construct frames of exponential functions for Cantor measures, introduced a new idea: to use Cuntz dilations to obtain orthonormal bases in spaces larger than the $L^2$-space of the given measure and then project them to construct Parseval frames. Even though the idea did not apply to the Middle Third Cantor set, a new class of Parseval frames was constructed for certain Cantor measures. The ideas were extended in \cite{DR18,DR16} and we present some of them here, in the context of Cuntz representations.

First, one needs to define the ground space: the Cantor measure, or more generally the self-affine measure. It is associated to a scale $R$ and a digit set $B$. For the Middle Third Cantor set $R=3$ and $B=\{0,2\}$. For the Jorgensen-Pedersen example in \cite{JP98}, $R=4$, $B=\{0,2\}$. 
\begin{definition}\label{defns}

For a given integer $R\geq 2$ and a finite set of integers $B$ with cardinality $|B|=:N,$ we define the \emph{affine iterated function system} (IFS) $\tau_{b}(x)=R^{-1}(x+b), x \in \br, b \in B.$ The \emph{self-affine measure} (with equal weights) is the unique probability measure $\mu = \mu(R, B)$ satisfying 
\begin{equation}\label{self-affine}
\int f\,d\mu=\frac1N\sum_{b\in B}\int f\circ\tau_b\,d\mu,\quad (f\in C_c(\br)).
\end{equation}
This measure is supported on the {\it attractor} $X_B$ which is the unique compact set that satisfies
$$
X_B= \bigcup_{b\in B} \tau_b(X_B).
$$
The set $X_B$ is also called the {\it self-affine set} associated with the IFS, and it can be described as 
$$X_B=\left\{\sum_{k=1}^\infty R^{-k}b_k : b_k\in B\right\}.$$

One can refer to \cite{Hut81} for a detailed exposition of the theory of iterated function systems. We say that $\mu = \mu(R,B)$ satisfies the {\it no overlap condition} if
$$
\mu(\tau_{b}(X_B)\cap \tau_{b'}(X_B))=0, \ \forall~b\neq b'\in B.
$$
For $\lambda\in \mathbb{R}$, define
$$e_\lambda(x)=e^{2\pi i\lambda x},~~ (x\in \mathbb{R}).$$

A \emph{frame} for a Hilbert space $H$ is a family $\{e_i\}_{i\in I}\subset H$ such that there exist constants $A,B>0$ such that for all $v\in H$,
$$A\|v\|^2\leq \sum_{i\in I}|\ip{v}{e_i}|^2\leq B\|v\|^2.$$
The largest $A$ and smallest $B$ which satisfy these inequalities are called the \emph{frame bounds}. The frame is called a \emph{Parseval} frame if both frame bounds are $1$. 
\end{definition}

Next, to construct the Parseval frame of exponential functions, one needs the dual set $L$ which acts as the starting point for the construction of the frequencies associated to the Fourier series. We make the following assumptions.

{\bf Assumptions 2.1.} Suppose that there exists a finite set $L \subset \mathbb{Z}$ with $0 \in L, |L|=:M$ and non-zero complex numbers $\left( \alpha_l \right)_{l \in L}$ such that the following properties are satisfied:
\begin{enumerate}
\item $\alpha_{0}=1.$
\item The matrix
\begin{equation}\label{matrix_T}
T := \frac{1}{\sqrt{N}} \left( e^{2 \pi i R^{-1}l \cdot b}\alpha_l \right)_{l \in L, b \in B}
\end{equation}
is an isometry, i.e., $T T^{*}=I_{N},$ i.e., its columns are orthonormal, which means that 
\begin{equation}
\frac{1}{N}\sum_{l\in L}|\alpha_l|^2e^{2\pi i R^{-1}l\cdot (b-b')}=\delta_{b,b'},\quad (b,b'\in B).
\label{eq2.2}
\end{equation}
\item The measure $\mu(R,B)$ has no overlap. 
\end{enumerate}

To formulate the result, we need some extra notations and definitions.

\begin{definition}\label{def1}

Let 
\begin{equation}
m_B(x)=\frac{1}{N}\sum_{b\in B}e^{2\pi i bx},\quad(x\in\br).
\label{eqmb}
\end{equation}

Since the measure $\mu(R,B)$ has no overlap, we can define the map $\mathcal R:X_B\rightarrow X_B$, by $$\mathcal R(x)=\tau_b^{-1}(x)=Rx-b,\mbox{ if }x\in\tau_b(X_B), b\in B.$$ 

A set $\mathcal M\subset \br$ is called {\it invariant} if for any point $t\in \mathcal M$, and any $l\in L$, if $\alpha_{l}m_B(R^{-1}(t-l))\neq 0$, then $g_{l}(t):=R^{-1}(t-l)\in \mathcal M$. $\mathcal M$ is said to be non-trivial if $\mathcal M\neq \{0\}$.  We call a finite {\it minimal invariant} set a {\it min-set}.

Note that 
\begin{equation}
\sum\limits_{l\in L} | \alpha_{l}  |^2 \ |m_{B}(g_{l} (t))|^{2}=1 \quad (t \in \mathbb{R}^d ),
\end{equation}
(see (3.2) in \cite[p.1615]{DR16}), and therefore, we can interpret the number $| \alpha_{l}  |^2 \ |m_{B}(g_{l} (t))|^{2}$ as the probability of transition from $t$ to $g_{l}(t)$, and if this number is not zero then we say that this {\it transition is possible in one step (with digit $l$)}, and we write $t\rightarrow g_{l}(t)$ or $t\stackrel{l}{\rightarrow}g_l(t)$. We say that the {\it transition is possible} from a point $t$ to a point $t'$ if there exist $t_0=t$, $t_1,\dots, t_n=t'$ such that $t=t_0\rightarrow t_1\rightarrow\dots\rightarrow t_n=t'$. The {\it trajectory} of a point $t$ is the set of all points $t'$ (including the point $t$) such that the transition is possible from $t$ to $t'$.

A {\it cycle} is a finite set $\{t_0,\dots,t_{p-1}\}$ such that there exist $l_0,\dots, l_{p-1}$ in $L$ such that $g_{l_0}(t_0)=t_1,\dots, g_{l_{p-1}}(t_{p-1})=t_{p}:=t_0$. Points in a cycle are called {\it cycle points}. 

A cycle $\{t_0,\dots, t_{p-1}\}$ is called {\it extreme} if $|m_B(t_i)|=1$ for all $i$; by the triangle inequality, since $0\in B$, this is equivalent to $t_i\cdot b\in \bz$ for all $b\in B$. 

For $k\in\bz$, we denote 
$$[k]:=\{k'\in\bz : (k'-k)\cdot R^{-1}b\in\bz, \mbox{ for all } b\in B\}.$$
\end{definition}

The next proposition gives some information about the structure of finite, minimal sets, which makes it easier to find such sets in concrete examples.

\begin{proposition}\label{pr3}\cite{DR18}
Let $\mathcal M $ be a non-trivial finite, minimal invariant set. Then, for every two points $t,t'\in \mathcal M $ the transition is possible from $t$ to $t'$ in several steps. In particular, every point in the set $\mathcal M $ is a cycle point. The set $\mathcal M $ is contained in the interval $\left[\frac{\min(-L)}{R-1},\frac{\max(-L)}{R-1}\right]$. 

If $t$ is in $\mathcal M $ and if there are two possible transitions $t\rightarrow g_{l_1}(t)$ and $t\rightarrow g_{l_2}(t)$, then $l_1\equiv l_2(\mod R)$.

Every point $t$ in $\mathcal M $ is an extreme cycle point, i.e., $|m_B(t)|=1$ and if $t\rightarrow g_{l_0}(t)$ is a possible transition in one step, then $[l_0]\cap L=\{l\in L : l\equiv l_0(\mod R)\}$ and
\begin{equation}
\sum_{l\in L,l\equiv l_0(\mod R)}|\alpha_l|^2=1.
\label{eqac}
\end{equation}

In particular $t\cdot b\in\bz$ for all $b\in B$. 
\end{proposition}

\medskip
\noindent
\begin{definition}
Let $c$ be an extreme cycle point in some finite minimal invariant set. A word $l_0 \dots l_{p-1}$ in $L$ is called a \emph{cycle word} for $c$ if $g_{l_{p-1}} \dots g_{l_{0}}(c)=c$
and $g_{l_k} \dots g_{l_{0}}(c) \neq c$ for $0 \leq k < p-1,$ and the transitions $c \rightarrow g_{l_{0}}(c) \rightarrow g_{l_{1}}g_{l_{0}}(c) \rightarrow \dots \rightarrow g_{l_{p-2}}\dots g_{l_{0}}(c)\rightarrow  g_{l_{p-1}}\dots g_{l_{0}}(c)=c$ are possible. 
\end{definition}

For every finite minimal invariant set $\mathcal{M},$ pick a point $c(\mathcal{M})$ in $\mathcal{M}$ and define $\Omega(c(\mathcal{M}))$ to be the set of finite words with digits in $L$ that do not end in a cycle word for $c(\mathcal{M})$, i.e., they are not of the form $\omega\omega_0$ where $\omega_0$ is a cycle word for $c$ and $\omega$ is an arbitrary word with digits in $L$.

\begin{theorem}\cite{DR18}\label{th1.6}
Suppose $(R,B,L)$ and $(\alpha_l)_{l\in L}$ satisfy the Assumptions 2.1. Then the set 
$$\Bigg\{ \left( \prod_{j=0}^{k} \alpha_{l_j} \right) e_{l_{0}+R l_{1} + \dots + R^k l_k + R^{k+1} c(\mathcal{M}) }  : l_0 \dots l_n \in \Omega(c(\mathcal{M})), \mathcal{M}~\textnormal{is a min-set}   \Bigg\}$$
is a Parseval frame for $L^2(\mu(R,B)).$
\end{theorem}

Here we will formulate these results, in our context, of row coisometries, Cuntz representations and random walks on graphs. This will give us a better understanding of the structure associated to the Fourier series on self-affine measures. 

\begin{definition}\label{def1.4}
We denote by $\Omega$ the set of all finite words with letters in $\Lambda$, including the empty word denoted $\ty$,
$$\Omega=\left\{ \lambda_1\dots \lambda_m : \lambda_1,\dots, \lambda_m\in  \Lambda , m\geq 0\right\}.$$
(In literature, the notation $\Omega=\Lambda^*$ is often used, as the free monoid generated by $\Lambda$.)

For $\lambda=\lambda_1\dots \lambda_m\in\Omega$, we denote by $|\lambda|=m$, the length of $\lambda$. 

We use the notation, for $\lambda_1\dots\lambda_m\in\Omega$, 
$$V_{\lambda_1\lambda_2\dots \lambda_m}=V_{\lambda_1}V_{\lambda_2}\dots V_{\lambda_m},\quad S_{\lambda_1\dots \lambda_m}=S_{\lambda_1}\dots S_{\lambda_m}.$$
\end{definition}

\begin{theorem}\label{thm.7}
Define the operators $(V_l)_{l\in L}$ on $L^2(\mu(R,B))$ by 
\begin{equation}
V_lf(x)=\alpha_le_l(x)f(\mathcal R(x)),\quad (x\in X_B,f\in L^2(\mu(R,B)), l\in L).
\label{eqm.7.1}
\end{equation}
Then $(V_l)_{l\in L}$ is a row coisometry. 

Let $\mathcal M$ be a min-set and let 
$$\mathcal K_{\mathcal M}=\Span\{ e_t : t\in\mathcal M\}.$$
Then $\mathcal K_{\mathcal M}$ is invariant for $V_l^*$, $l\in L$ and 
\begin{equation}
V_l^* e_t=\cj \alpha_l m_B(g_l(t))e_{g_l(t)},\quad (t\in\br).
\label{eqm.7.2}
\end{equation}
In addition, for all $t\in\mathcal M$, $m_B(g_l(t))=1$, if the transition $t\arr l g_l(t)$ is possible and 
\begin{equation}
\sum_{\stackrel{l\in L, t\arr l g_l(t)}{\mbox{ \tiny{is possible}}}}|\alpha_l|^2=1.
\label{eqm.7.3}
\end{equation}
The exponential functions $\{e_t : t\in\mathcal M\}$ form an orthonormal basis for $\mathcal K_{\mathcal M}$ and the spaces $\mathcal K_{\mathcal M}$ are mutually orthogonal.

Let 
$$\mathcal H_{\mathcal M}=\Span\{ V_{l_0\dots l_k}e_t : k\geq 0, l_0,\dots, l_k\in L, t\in\mathcal M\}.$$
\begin{equation}
 V_{l_0\dots l_k}e_t=\left( \prod_{j=0}^{k} \alpha_{l_j} \right)e_{l_0+Rl_1+\dots+R^kl_k+R^{k+1}t}, \quad (k\geq 0, l_0,\dots, l_k\in L, t\in\mathcal M).
\label{eqm.7.3.1}
\end{equation}

\begin{equation}
\Span\{\mathcal H_{\mathcal M}: {\mathcal M\mbox{ min-set}}\}=L^2(\mu(R,B)).
\label{eqm.7.4}
\end{equation}

The coisometry $(P_{\mathcal K_{\mathcal M}}V_lP_{\mathcal K_{\mathcal M}})_{l\in L}$ on $\mathcal K$ is isomorphic to the coisometry associated to the graph with vertices $\mathcal V=\mathcal M$ and transition weights $\alpha_{t,l}:=\cj \alpha_l $, $t\in \mathcal M$, $l\in L$, through the linear map defined by $\bc[\mathcal M]\ni \vec t\mapsto e_t\in \mathcal K_{\mathcal M}$. 

The Cuntz dilation of this coisometry is irreducible.

\end{theorem}

\begin{proof}
Let $\mu=\mu(R,B)$. First, we compute the adjoints $V_l^*$. We have, using \eqref{self-affine}, for $f,g\in L^2(\mu)$:
$$\ip{V_l f}{g}=\int \alpha_l e^{2\pi i lx}f(\mathcal Rx)\cj g(x)\,d\mu=\frac{1}{N}\sum_{b\in B}\int\alpha_l e^{2\pi i l\tau_b(x)}f(\mathcal R\tau_b(x))\cj g(\tau_b(x))\,d\mu$$
$$=\ip{f}{\frac{1}{N}\sum_{b\in B}\cj\alpha_le_{-l}\circ\tau_b\cdot g\circ\tau_b}.$$
Therefore,
\begin{equation}
V_l^*g=\frac{1}{N}\sum_{b\in B}\cj\alpha_le_{-l}\circ\tau_b\cdot g\circ\tau_b,\quad (g\in L^2(\mu)).
\label{eqm.7.5}
\end{equation}
Then, for $f\in L^2(\mu)$ and $x\in X_B$, suppose $x\in\tau_{b'}(X_B)$, and we have, using the assumptions, 
$$\sum_{l\in L}V_lV_l^*f(x)=\sum_{l\in L}\alpha_le^{2\pi ilx}\frac{1}{N}\sum_{b\in B}\cj\alpha_le^{-2\pi il\frac{Rx-b'+b}{R}}f\left(\frac{Rx-b'+b}{R}\right)$$
$$=\sum_{b\in B}f\left(x+\frac{b-b'}R\right)\frac{1}{N}\sum_{l\in L}|\alpha_l|^2e^{2\pi il\frac{b-b'}{R}}=\sum_{b\in B}f\left(x+\frac{b-b'}R\right)\delta_{bb'}=f(x).$$
This shows that $(V_l)_{l\in L}$ is a row coisometry. 

We compute, for $t\in\mathcal \br$, 
$$V_l^*e_t(x)=\frac{1}{N}\sum_{b\in B}\cj\alpha_l e^{-2\pi i l\frac{x+b}{R}}e^{2\pi it\frac{x+b}{R}}=\cj\alpha_l\left(\frac{1}{N}\sum_{b\in B}e^{2\pi  i b\cdot\frac{t-l}{R}}\right)e^{2\pi ix\cdot\frac{t-l}R}=\cj\alpha_lm_B(g_l(t))e_{g_l(t)}(x).$$
This implies \eqref{eqm.7.2}.

Now let $\mathcal M$ be a min-set. For $t\in\mathcal M$, we have $V_l^*e_t=\cj\alpha_lm_B(g_l(t))e_{g_l(t)}$. If the transition $t\arr l g_l(t)$ is not possible, that means that $\cj \alpha_lm_B(g_l(t))=0$ so $V_l^*e_t=0$. If the transition $t\arr l g_l(t)$ is possible, then $g_l(t)\in\mathcal M$ so $V_l^*e_t\in\mathcal K_{\mathcal M}$. Thus $\mathcal K_{\mathcal M}$ is invariant for $V_l^*$. 

For $t\in\mathcal M$, if the transition $t\arr l g_l(t)$ is possible, then $g_l(t)\in\mathcal M$, and, by Proposition \ref{pr3}, $bg_l(t)\in\bz$ for all $b\in B$ so $m_B(g_l(t))=1$. Also, from the same Proposition, we have that if the transitions $t\arr {l_1} g_{l_1}(t)$ and $t\arr{l_2} g_{l_2}(t)$ are possible, then $l_1\equiv l_2(\bmod R)$. Conversely, if the transition $t\arr{l_1} g_{l_1}(t)$ is possible, and $l_1\equiv l_2(\bmod R)$, then $m_B(g_{l_2}(t))=\frac{1}{N}\sum_{b\in B}e^{2\pi i b\frac{t-l_2}{R}}=\frac{1}{N}\sum_{b\in B}e^{2\pi i b\frac{t-l_1}{R}}=m_B(g_{l_1}(t))\neq 0$, so the transition $t\arr{l_2} g_{l_2}(t)$ is possible (note that we assumed that the numbers $\alpha_l$ are all non-zero). Therefore, using again Proposition \ref{pr3}, fixing some $l_0\in L$ such that the transition $t\arr {l_0} g_{l_0(t))}$ is possible, 
$$\sum_{\stackrel{l\in L, t\arr l g_l(t)}{\mbox{ \tiny{is possible}}}}|\alpha_l|^2|m_B(g_l(t))|^2=\sum_{l\equiv l_0(\bmod R)}|\alpha_l|^2=1.$$

The fact that the functions $e_t$, $t\in \mathcal M$, $\mathcal M$ min-set, are mutually orthogonal is in the proof of \cite[Theorem 1.6]{DR18}. Equation \eqref{eqm.7.3.1} follows from a simple computation, using the fact that, by Proposition \ref{pr3}, for all $t\in\mathcal M$, $bt\in\bz$ and so $t$ is a period for $m_B$, i.e., $m_B(x+kt)=m_B(t)$ for all $x\in\br$, $k\in\bz$. The relation \eqref{eqm.7.4} follows from Theorem \ref{th1.6}. 

For $t\in\mathcal M$, by \eqref{eqm.7.2}, 
$$P_{\mathcal K_{\mathcal M}}V_l^*P_{\mathcal K_{\mathcal M}}e_t=\cj\alpha_lm_B(g_l(t))P_{\mathcal K_{\mathcal M}}e_{g_l(t)}=\left\{\begin{array}{cc}
\cj\alpha_l e_{g_l(t)},&\mbox{ if the transition $t\arr lg_l(t)$ is possible},\\
0,&\mbox{ otherwise.}\end{array}\right.$$
So , the coisometry $(P_{\mathcal K_{\mathcal M}}V_l^*P_{\mathcal K_{\mathcal M}})_{l\in L}$ on $\mathcal K_{\mathcal M}$ is isometric to the given graph coisometry. 

To see that the associated Cuntz dilation is irreducible, we use Corollary \ref{cor5.10} and show that the random walk is connected and separating. 

The fact that the random walk is connected follows from the minimality of $\mathcal M$ and Proposition \ref{pr3}. To check that the random walk is separating take $t\neq t'$ in $\mathcal M$. Let $l_0,\dots,l_n$ in $L$ such that both transitions $t\arr{l_0\dots l_n} g_{l_n\dots l_0}(t)$ and $t'\arr{l_0\dots l_n} g_{l_n\dots l_0}(t')$ are possible. Then $g_{l_n\dots l_0}(t)$ and $g_{l_n\dots l_0}(t')$ are in $\mathcal M$. But, the maps $g_l$ are contractions so $\lim_n \mbox{dist}(g_{l_n\dots l_0}(t),g_{l_n\dots l_0}(t'))=0$. Since $\mathcal M$ is finite, for $n$ large enough, we get that $g_{l_n\dots l_0}(t)=g_{l_n\dots l_0}(t')$, but that means $t=t'$, a contradiction. Thus, for $n$ large enough, we cannot have that both transitions $t\arr{l_0\dots l_n} g_{l_n\dots l_0}(t)$ and $t'\arr{l_0\dots l_n} g_{l_n\dots l_0}(t')$ are possible. So, the random walk is separating.

\end{proof}

\section{General results}

\begin{proposition}\label{pr2.1}
Let $(V_{\lambda})_{\lambda\in\Lambda}$ be a row coisometry on the Hilbert space $\mathcal K$ and let $(S_{\lambda})_{\lambda\in\Lambda}$ be its Cuntz dilation on the Hilbert space $\mathcal H$. Define the subspaces 
\begin{equation}
\mathcal K_m=\Span\{S_{\lambda_1\dots \lambda_m}k : k\in\mathcal K, \,\lambda_1,\dots \lambda_m\in \Lambda \},\quad (m\in\bn), \quad\mathcal K_0=\mathcal K.
\label{eq2.1.1}
\end{equation}
\begin{enumerate}
	\item[(i)] $\{\mathcal K_m\}$ is an increasing sequence of subspaces and 
	$$\overline{\bigcup_{m\in\bn}\mathcal K_m}=\mathcal H.$$
	\item[(ii)] For each $m\geq 0$ and $\lambda\in \Lambda $, 
$$S_{\lambda}^*\mathcal K_{m+1}\subseteq \mathcal K_m.$$
\item[(iii)] For $k\in\mathcal K_m$, and each $n\leq m$, we have that the representation 
$$k=\sum_{\lambda\in\Omega, |\lambda|=n}S_\lambda k_\lambda,$$
is unique with $k_\lambda\in\mathcal H$, and moreover, $k_\lambda$ is given by $k_\lambda=S_\lambda^*k\in\mathcal K_{m-n}$. 
\item[(iv)] Let $P_{\mathcal K_m}$ be the orthogonal projection onto $\mathcal K_m$. Then 
$$P_{\mathcal K_m}=\sum_{\lambda_1,\dots,\lambda_m\in \Lambda }S_{\lambda_1\dots \lambda_m}P_{\mathcal K}S_{\lambda_1\dots \lambda_m}^*.$$

\end{enumerate}

\end{proposition}

\begin{proof}
(i) Let $S_{\lambda_1\dots \lambda_m}k\in\mathcal K_m$, with $k\in\mathcal K$. Then, using the Cuntz relations
$$S_{\lambda_1\dots \lambda_m}k=S_{\lambda_1\dots \lambda_m}\sum_{\lambda\in\Lambda}S_{\lambda}S_{\lambda}^*k=\sum_{\lambda\in\Lambda}S_{\lambda_1\dots \lambda_m}S_{\lambda}(S_{\lambda}^*k).$$
Since $S_{\lambda}^*k=V_{\lambda}^*k\in\mathcal K$ for each $\lambda$, it follows that $S_{\lambda_1\dots \lambda_m}k\in\mathcal K_{m+1}$, so $\mathcal K_m\subseteq \mathcal K_{m+1}.$

The density of the union follows from the fact that $\mathcal K$ is cyclic for the representation.

(ii) Let $S_{\lambda_1\dots \lambda_{m+1}}k\in\mathcal K_{m+1}$. Then $S_{\lambda}^*S_{\lambda_1\dots \lambda_{m+1}}k=\delta_{\lambda_1}^\lambda S_{\lambda_2\dots \lambda_{m+1}}k\in\mathcal K_m$.

(iii) Assume we have two such representations

$$k=\sum_{\lambda\in\Omega, |\lambda|=n}S_\lambda k_\lambda=\sum_{\lambda\in\Omega, |\lambda|=n}S_\lambda k_\lambda'.$$
Then, using the orthogonality of the ranges of the isometries
$$0=\|k-k\|^2=\sum_{|\lambda|=n}\|S_\lambda(k_\lambda-k_\lambda')\|^2=\sum_{|\lambda|=n}\|k_\lambda-k_\lambda'\|^2.$$
Therefore $k_\lambda= k_\lambda'$ for all $|\lambda|=n$. 

Further, for all $n\leq m$ and $|\lambda|=n$, by inducting on (ii) we have $k_\lambda=S_{\lambda}^*k\in \mathcal K_{m-n}$, and, using the Cuntz relations, 
$$k=\sum_{|\lambda|=n}S_\lambda S_{\lambda}^*k=\sum_{|\lambda|=n}S_\lambda k_\lambda.$$

(iv) Denote 
$$T=\sum_{\lambda_1\dots \lambda_m}S_{\lambda_1\dots \lambda_m}P_{\mathcal K}S_{\lambda_1\dots \lambda_m}^*.$$

If $h\in\mathcal K_m$, then, by (ii), we have that $S_{\lambda_1\dots \lambda_m}^*\in \mathcal K$, so $P_{\mathcal K}S_{\lambda_1\dots \lambda_m}^*h=S_{\lambda_1\dots \lambda_m}^*h$. Therefore 
$$Th=\sum_{\lambda_1\dots \lambda_m} S_{\lambda_1\dots \lambda_m}\left(P_{\mathcal K}S_{\lambda_1\dots \lambda_m}^*h\right)=\sum_{\lambda_1\dots \lambda_m}S_{\lambda_1\dots \lambda_m}S_{\lambda_1\dots \lambda_m}^*h=h.$$
Now, if $h\perp\mathcal K_m$, then we have that for each $k\in\mathcal K$, 
$$\ip{S_{\lambda_1\dots \lambda_m}^*h}{k}=\ip{h}{S_{\lambda_1\dots \lambda_m}k}=0.$$
This means that $P_{\mathcal K}S_{\lambda_1\dots \lambda_m}^*h=0$, and therefore $Th=0$ for $h\perp \mathcal K_m$ and $Th=h$ for $h\in\mathcal K_m$. In conclusion $T=P_{\mathcal K_m}$.
\end{proof}

\begin{proposition}\label{pr2.2}
With the notations in Theorem \ref{th1.1}, let $T\in\mathcal B(\mathcal K)^{\bm \sigma}$ and let $A_T$ be the associated operator on $\mathcal H$ which commutes with the representation $(S_{\lambda})_{\lambda\in\Lambda}$, defined uniquely by the property $T=P_{\mathcal K}A_TP_{\mathcal K}$. Then, for $m\in\bn$, 
\begin{equation}
P_{\mathcal K_m}A_TP_{\mathcal K_m}=\sum_{|\lambda|=m}S_\lambda TS_\lambda^*.
\label{eq2.2.1}
\end{equation}
Also, $P_{\mathcal K_m}A_TP_{\mathcal K_m}\xi\arr{m\rightarrow\infty}A_T\xi$ for all $\xi\in \mathcal H$. 
\end{proposition}

\begin{proof}
Using Proposition \ref{pr2.1}(iv), and the fact that $A_T$ commutes with the representation, we have
$$P_{\mathcal K_m}A_TP_{\mathcal K_m}=\sum_{|\lambda|=m}\sum_{|\lambda'|=m}S_\lambda P_{\mathcal K}S_{\lambda}^*A_TS_{\lambda'}P_{\mathcal K}S_{\lambda'}^*
=\sum_{|\lambda|=m}\sum_{|\lambda'|=m}S_\lambda P_{\mathcal K}S_{\lambda}^*S_{\lambda'}A_TP_{\mathcal K}S_{\lambda'}^*$$$$=\sum_{|\lambda|=m}S_\lambda P_{\mathcal K}A_TP_{\mathcal K}S_{\lambda}^*=\sum_{|\lambda|=m}S_\lambda TS_{\lambda}^*.$$

The following Lemma is well known, see e.g., \cite[Section I.6]{Dav96}.
\begin{lemma}\label{lem2.3}
Let $(A_n)_{n\in\bn}$, $(B_n)_{n\in\bn}$ be sequences of bounded operators on some Hilbert space $\mathcal H$. Assume that $A_n\rightarrow A$ and $B_n\rightarrow B$ in the Strong Operator Topology (SOT), i.e., $A_n\xi\rightarrow A\xi$, $B_n\xi\rightarrow B\xi$, for all vectors $\xi\in \mathcal H$. Then $A_nB_n\rightarrow AB$ in the SOT. 
\end{lemma}

Since $P_{\mathcal K_m}\rightarrow I_{\mathcal H}$ in the SOT, with Lemma \ref{lem2.3}, we obtain that $P_{\mathcal K_m}A_TP_{\mathcal K_m}\rightarrow A_T$ in SOT. 
\end{proof}

\begin{proposition}\label{pr2.4}
With the notations as in Theorem \ref{th1.1}, let $T,T'\in\mathcal B(\mathcal K)^{\bm \sigma}$ and let $A$ and $A'$, respectively, be the associated operators in the commutant of the Cuntz dilation, so $T=P_{\mathcal K}AP_{\mathcal K}$ and $T'=P_{\mathcal K}A'P_{\mathcal K}$. Then $AA'$ is also in the commutant of the Cuntz dilation, so $T\ast T':=P_{\mathcal K}AA'P_{\mathcal K}$ is an element in $\mathcal B(\mathcal K)^{\bm \sigma}$. Then, 
\begin{equation}
(P_{\mathcal K_m}AP_{\mathcal K_m})(P_{\mathcal K_m}A'P_{\mathcal K_m})\xi\rightarrow AA'\xi,\quad (\xi\in\mathcal H),
\label{eq2.4.1}
\end{equation}
and 
\begin{equation}
(T\ast T')\xi=\lim_{m\rightarrow\infty}\sum_{|\lambda|=m}V_\lambda TT'V_\lambda^*\xi,\quad (\xi\in\mathcal K).
\label{eq2.4.2}
\end{equation}
\end{proposition}

\begin{proof}
The limit in \eqref{eq2.4.1} follows from Proposition \ref{pr2.2} and Lemma \ref{lem2.3}. Then, with Lemma \ref{lem2.3}, \eqref{eq2.2.1}, for $\xi\in \mathcal{K} $,
$$P_{\mathcal K}AA'P_{\mathcal K}\xi=\lim_{m\rightarrow\infty}P_{\mathcal K}P_{\mathcal K_m}AP_{\mathcal K_m}P_{\mathcal K_m}A'P_{\mathcal K_m}P_{\mathcal K}\xi=
\lim_{m\rightarrow\infty}P_{\mathcal K}\sum_{|\lambda|=m}S_\lambda TS_\lambda^*\sum_{|\lambda'|=m}S_{\lambda'} T'S_{\lambda'}^*P_{\mathcal K}\xi$$
$$=\lim_{m\rightarrow\infty}P_{\mathcal K}\sum_{|\lambda|=m}\sum_{|\lambda'|=m}S_\lambda TS_\lambda^*S_{\lambda'} T'S_{\lambda'}^*P_{\mathcal K}\xi=\lim_{m\rightarrow\infty}P_{\mathcal K}\sum_{|\lambda|=m}S_\lambda TT'S_\lambda^*P_{\mathcal K}\xi$$
$$=\lim_{m\rightarrow\infty}\sum_{|\lambda|=m}P_{\mathcal K}S_\lambda TT'V_\lambda^*\xi=\lim_{m\rightarrow\infty}\sum_{|\lambda|=m}V_\lambda TT'V_\lambda^*\xi.$$
\end{proof}

\section{An explicit construction of the Cuntz dilation associated to a random walk}

Consider now, as in Definition \ref{def1.3} a directed graph $\mathcal G=(\mathcal V, \mathcal E)$, with edges labeled from a finite set $\Lambda$, $|\Lambda|=N$. Recall, that we assume that for each vertex $i$, different labels $\lambda_1$, $\lambda_2$, lead, from $i$ to different vertices, so, if $i\stackrel{\lambda_1}{\rightarrow}j_1$, and $i\stackrel{\lambda_2}{\rightarrow}j_2$, then $j_1\neq j_2$. We write $j=i\cdot\lambda$ if $i\stackrel{\lambda}{\rightarrow}j$. Also, we assume that for each vertex $i$ there is at most one edge coming into the vertex $i$ with label $\lambda$. 

Each edge has an associated weight $|\alpha_{i,\lambda}|^2$ defined by some complex number $\alpha_{i,\lambda}$, $\alpha_{i,\lambda}=0$ in the case when there is no edge from $i$, labeled $\lambda$, and 
\begin{equation}
\sum_{\lambda\in\Lambda}|\alpha_{i,\lambda}|^2=1,\quad (i\in\mathcal V).
\label{eqde3.3.1}
\end{equation}

We recall that we define the operators $(V_\lambda)_{\lambda\in\Lambda}$ on the Hilbert space $\mathcal K=l^2[\mathcal V]$,

\begin{equation}
V_\lambda^*(\vec i)=\left\{\begin{array}{cc}
	\alpha_{i,\lambda}\vec j,&\mbox{ if }i\stackrel{\lambda}{\rightarrow}j\\
	0,&\mbox{ otherwise.}
\end{array}\right.
\label{eqde3.3.2}
\end{equation}

\begin{proposition}\label{pr3.1}
The operators $(V_\lambda)_{\lambda\in\Lambda}$ form a row coisometry. 
\end{proposition}

\begin{proof}
A simple computation shows that 
\begin{equation}
V_\lambda(\vec j)=\left\{\begin{array}{cc}
\cj\alpha_{i,\lambda} \vec i,&\mbox{ if }i\stackrel{\lambda}{\rightarrow}j,\\
0,&\mbox{ otherwise.}
\end{array}\right.
\label{eq3.3.3}
\end{equation}
Then, for $i\in\mathcal V$,
$$\sum_{\lambda\in\Lambda}V_\lambda V_\lambda^*\vec i=\sum_{\lambda\in\Lambda}\alpha_{i,\lambda} V_\lambda(\vec{i\cdot\lambda})=\sum_{\lambda\in\Lambda}|\alpha_{i,\lambda}|^2\vec i=\vec i.$$
\end{proof}

Since $(V_\lambda)_{\lambda\in\Lambda}$ is a row  coisometry, by Theorem \ref{th1.1}, it has a unique Cuntz dilation. In this section we will give an explicit construction of the Cuntz dilation associated to this random walk, under a certain mild assumption. 

We will need some notation. Recall that, for a word $\lambda=\lambda_1\lambda_2\dots\lambda_l$, we define 
\begin{equation}
\alpha_{i,\lambda}=\alpha_{i,\lambda_1}\alpha_{i\cdot\lambda_1,\lambda_2}\dots\alpha_{i\cdot\lambda_1\cdot\lambda_2\cdots\lambda_{l-1},\lambda_l}.
\label{eq3.3.4}
\end{equation}

Note that $|\alpha_{i,\lambda}|^2$ is the probability that, starting from the vertex $i$, the random walk follows the labels $\lambda_1,\lambda_2,\dots,\lambda_l$. 

For each vertex $i$, let $\Lambda_i$ be the set of labels that originate from $i$, 
\begin{equation}
\Lambda_i:=\left\{\lambda\in\Lambda : \mbox{ There exists $j$ such that }i\stackrel{\lambda}{\rightarrow}j\right\}.
\label{eq3.3.5}
\end{equation}
For each vertex $j$, let $\Lambda^j$ be the set of all labels that end in $j$,
\begin{equation}
\Lambda^j:=\left\{\lambda\in\Lambda : \mbox{ There exists $i$ such that }i\arr{\lambda}j\right\}.
\label{eq3.3.6}
\end{equation}

Further, we define $n_i:=|\Lambda_i|$ and $n^j:=|\Lambda^j|$. 

To construct the Cuntz dilation, first we will construct, for each vertex $i$, some unitary matrix that has, as the first column, the weights $(\alpha_{i,\lambda})_{\lambda\in\Lambda_i}$. Indeed, since 
$$\sum_{\lambda\in\Lambda_i}|\alpha_{i,\lambda}|^2=1,$$
we may create the following unitary matrix (not necessarily unique)
\begin{equation}
C_i= \begin{bmatrix}
	\alpha_{i,\lambda_0}&c_{i,\lambda_0}^1&c_{i,\lambda_0}^2&\cdots&c_{i,\lambda_0}^{n_i-1}\\
	\alpha_{i,\lambda_1}&c_{i,\lambda_1}^1&c_{i,\lambda_1}^2&\cdots&c_{i,\lambda_1}^{n_i-1}\\
	\vdots&\vdots&\vdots&\ddots&\vdots\\
	\alpha_{i,\lambda_{n_i-1}}&c_{i,\lambda_{n_i-1}}^1&c_{i,\lambda_{n_i-1}}^2&\cdots&c_{i,\lambda_{n_i-1}}^{n_i-1}
\end{bmatrix}=\left[c_{i,\lambda}^k\right]_{\lambda\in\Lambda_i}^{k=0,\dots,n_i-1},
\label{eq3.3.7}
\end{equation}
where we adopt the notation $c_{i,\lambda_j}^0:=\alpha_{i,\lambda_j}$.

To define the Hilbert space of the Cuntz dilation, we will use the set $\Omega_N^*$, defined as the set of finite words over the alphabet $\{0,1,\dots, N-1\}$ not ending in 0, including the empty word.

For a digit $k\in\{0,\dots,N-1\}$ and a word $w\in\Omega_N^*$, define $kw\in\Omega_N^*$ as the concatenation of $k$ and $w$. We make the {\bf important convention}:
$$0\ty=\ty.$$
Additionally, we define the ``inverse concatenation'', $\setminus:\Omega_N^*\times\{0,\dots,N-1\}\rightarrow \Omega_N^*\cup\{\nul\}$, 
$$\setminus(w,k):=w\setminus k:=\left\{\begin{array}{cc}
w',&\mbox{ if }kw'=w,\\
\nul,&\mbox{ otherwise.}\end{array}\right.$$
Note that $\ty\setminus 0=\ty$ and $\ty\setminus k=\nul$ for $k\neq 0$. 

Expanding our notation, we define 
$$\lambda\cdot j:=\left\{\begin{array}{cc}
i,&\mbox{ if }i\arr{\lambda}j\\
\nul,&\mbox{ otherwise.}\end{array}\right.
$$
$$i\cdot\lambda:=\left\{\begin{array}{cc}
j,&\mbox{ if }i\arr{\lambda}j\\
\nul,&\mbox{ otherwise.}\end{array}\right.
$$
\medskip
We define the Hilbert space of the dilation as 
$$\mathcal H=l^2[\mathcal V\times\Omega_N^*]=\Span\{(i,w) : i \mbox{ vertex in }\mathcal V, w\in\Omega_N^*\}.$$
We identify $\mathcal K$ with 
$$\mathcal K=\Span\{(i,\ty) : i\in\mathcal V\}.$$

\begin{remark} 
For a fixed $w\in\Omega_N^*$, $w\setminus k\neq \nul$ for exactly one digit $k\in\{0,\dots,N-1\}$. 
We use the notation 
$$\delta_{w}^{w'}=\left\{\begin{array}{cc} 1,&\mbox{ if }w=w'\\ 0,&\mbox{ if }w\neq w',\end{array}\right. (w,w'\in\Omega).$$
We define also $\delta_w^{w'}$ when $w=\nul$ or $w'=\nul$, by making the conventions:
$$(\nul, \nul) = (\nul,w)=(i,\nul)=0 \in \mathcal{H},\quad \delta_i^{\nul}=\delta_w^{\nul}=\delta_{\nul}^{\nul}=0,$$
for all vertices $i$ and all words $w$.

Note also that 
$$\delta_{kw}^{w'}=\delta_{w}^{w'\setminus k},$$
$$\delta_{i}^{i'}=\delta_{i\cdot\lambda}^{i'\cdot\lambda},\quad(\lambda\in\Lambda_i),\quad 
\delta_{j}^{j'}=\delta_{\lambda\cdot j}^{\lambda\cdot j'},\quad(\lambda\in\Lambda^j).$$
\end{remark}

We will make the following assumption:
\begin{equation}
\sum_{i\in\mathcal V}(N-n_i)=\sum_{j\in\mathcal V}(N-n^j).
\label{eq3.3.8}
\end{equation}

In this case, consider the two sets
$$\left\{(j,\lambda) : j\in\mathcal{V}, \lambda\in\Lambda\setminus\Lambda^j\right\},\quad \left\{(i,k) : i\in\mathcal V, n_i\leq k\leq N-1\right\}.$$
Note that the first set has cardinality $\sum_{j\in\mathcal V}(N-n^j)$, and the second set has cardinality $\sum_{i\in\mathcal V}(N-n_i)$. Therefore, under assumption \eqref{eq3.3.8}, the two sets have equal cardinality, so there is a bijection between them
$$\varphi:\left\{(j,\lambda) : j\in\mathcal{V}, \lambda\in\Lambda\setminus\Lambda^j\right\}\rightarrow \left\{(i,k) : i\in\mathcal V, n_i\leq k\leq N-1\right\}.$$ We define:
\begin{equation}
F=\pi_1\circ \varphi, G=\pi_2\circ\varphi, \tilde F=\pi_1\circ\varphi^{-1},\tilde G=\pi_2\circ\varphi^{-1},
\label{eq3.3.9}
\end{equation}
where $\pi_1,\pi_2$ are the projections onto the first and second component. 

\begin{remark}\label{rem3.3}
Note that, when the graph is finite, the assumption \eqref{eq3.3.8} holds. Indeed, each edge is completely and uniquely determined by its starting vertex $i$ and a label in $\Lambda_i$. Therefore the number of edges is $\sum_{i\in\mathcal V}n_i$. On the other hand each edge is completely and uniquely determined by its end vertex $j$ and a label in $\Lambda^j$, therefore the number of edges is also $\sum_{j\in\mathcal V}n^j$. Since the graph is finite, this implies \eqref{eq3.3.8}.
\end{remark}

\begin{theorem}\label{th3.4}
Suppose the assumption \eqref{eq3.3.8} holds. Define the operators $(S_\lambda)_{\lambda\in\Lambda}$ on $\mathcal H$ by 
\begin{equation}
S_\lambda(j,w)=\left\{\begin{array}{cc}
\sum_{k=0}^{n_i-1}\cj{c_{i,\lambda}^k}(i,kw),&\mbox{ if }i\arr{\lambda}j,\\
(F(j,\lambda),G(j,\lambda)w),&\mbox{ otherwise.}
\end{array}\right.
\label{eq3.3.10}
\end{equation}
Then $(S_\lambda)_{\lambda\in\Lambda}$ is the Cuntz dilation of the row coisometry $(V_\lambda)_{\lambda\in\Lambda}$ in \eqref{eq3.3.3}. Moreover, the adjoint operators are
\begin{equation}
S_\lambda^*(i,k'w)=\left\{\begin{array}{cc}
\sum_{k=0}^{n_i-1}c_{i,\lambda}^k(j,k'w\setminus k),&\mbox{ if }i\arr{\lambda}j,\\
(\tilde F(i,k'),w),&\mbox{ if }\lambda=\tilde G(i,k'),\\
0,&\mbox{ otherwise.}
\end{array}
\right.
\label{eq3.3.11}
\end{equation}
\end{theorem}

\begin{proof}
First, we compute the adjoints. Let $\lambda\in\Lambda$, $j,i'\in\mathcal V$, $w,w'\in\Omega_N^*$. Case 1: $\lambda\in\Lambda^j$; then denote $i\arr\lambda j$.
$$\ip{(j,w)}{S_\lambda^*(i',w')}=\ip{S_\lambda(j,w)}{(i',w')}=\ip{\sum_{k=0}^{n_i-1}\cj{c_{i,\lambda}^k}(i,kw)}{(i',w')}$$
$$=\sum_{k=0}^{n_i-1}\cj{c_{i,\lambda}^k}\delta_i^{i'}\delta_{kw}^{w'}=\sum_{k=0}^{n_{i'}-1}\cj{c_{i',\lambda}^k}\delta_{j}^{i'\cdot\lambda}\delta_w^{w'\setminus k}=\ip{(j,w)}{\delta_{j}^{i'\cdot \lambda}\sum_{k=0}^{n_{i'}-1}c_{i',\lambda}^k(i'\cdot\lambda,w'\setminus k)}.$$
Case 2: $\lambda\not\in\Lambda^j$. Then
$$\ip{(j,w)}{S_\lambda^*(i',k'w')}=\ip{S_\lambda(j,w)}{(i',k'w')}=\ip{(F(j,\lambda),G(j,\lambda)w)}{(i',kw')}$$
$$=\delta_{F(j,\lambda)}^{i'}\delta_{G(j,\lambda)}^{k'}\delta_{w}^{w'}=\delta_{\varphi(j,\lambda)}^{(i',k')}\delta_{w}^{w'}=\delta_{j}^{\tilde F(i',k')}\delta_{\lambda}^{\tilde G(i',k')}\delta_w^{w'}=\ip{(j,w)}{\delta_\lambda^{\tilde G(i',k')}(\tilde F(i',k'),w')}.$$
This proves \eqref{eq3.3.11}.

Now, we verify the Cuntz relations. Let $j\in\mathcal V$, $w\in\Omega_N^*$, $\lambda,\lambda'\in\Lambda$. Case 1: $\lambda\in\Lambda^j$. Then we take $i$ such that $i\arr\lambda j$. Then 
$$S_{\lambda'}^*S_{\lambda}(j,w)=S_{\lambda'}^*\sum_{k=0}^{n_i-1}\cj{c_{i,\lambda}^k}(i,kw)=\sum_{k'=0}^{n_i-1}\sum_{k=0}^{n_i-1}\cj{c_{i,\lambda}^k}c_{i,\lambda'}^{k'}(i\cdot\lambda',kw\setminus k')$$
$$=\sum_{k=0}^{n_i-1}\cj{c_{i,\lambda}^k}c_{i,\lambda'}^{k}(i\cdot\lambda',w)=\delta_\lambda^{\lambda'}(j,w),$$
where we used the orthogonality of the matrix $C_i$ in the last equality.

Case 2: $\lambda\not\in\Lambda^j$. Then, since $\varphi$ is a bijection, 
$$S_{\lambda'} ^* S_\lambda(j,w)=S_{\lambda'}^*(F(j,\lambda),G(j,\lambda)w)=\delta_{\lambda}^{\lambda'}(j,w).$$

Now, we check the second Cuntz relation. Let $i\in\mathcal V$, $k\in\{0,\dots,N-1\}$, $w\in\Omega_N^*$. Case 1: $k\geq n_i$. Then
$$\sum_{\lambda\in\Lambda}S_\lambda S_\lambda^*(i,kw)=\sum_{\lambda\in\Lambda_i}S_\lambda S_\lambda^*(i,kw)+\sum_{\lambda\not\in\Lambda_i}S_\lambda S_\lambda^*(i,kw)$$
$$=\sum_{\lambda\in\Lambda_i}S_\lambda\sum_{k'=0}^{n_i-1}c_{i,\lambda}^{k'}(i\cdot\lambda,kw\setminus k')+\sum_{\lambda\not\in\Lambda_i}S_\lambda \delta_{\lambda}^{\tilde G(i,k)}(\tilde F(i,k),w)$$
$$=0+S_{\tilde G(i,k)}(\tilde F(i,k),w)=(i,kw).$$

Case 2: $k<n_i$. Then
$$\sum_{\lambda\in\Lambda}S_\lambda S_\lambda^*(i,kw)=\sum_{\lambda\in\Lambda_i}S_\lambda S_\lambda^*(i,kw)+\sum_{\lambda\not\in\Lambda_i}S_\lambda S_\lambda^*(i,kw)$$
$$=\sum_{\lambda\in\Lambda_i}S_\lambda\sum_{k'=0}^{n_i-1}c_{i,\lambda}^{k'}(i\cdot\lambda,kw\setminus k')+0=\sum_{\lambda\in\Lambda_i}c_{i,\lambda}^kS_\lambda(i\cdot \lambda,w)$$
$$=\sum_{\lambda\in\Lambda_i}\sum_{k'=0}^{n_i-1}\cj{c_{i,\lambda}^{k'}}c_{i,\lambda}^k(i,k'w)=\sum_{k'=0}^{n_i-1}\delta_{k}^{k'}(i,k'w)=(i,kw).$$
Thus, the Cuntz relations are satisfied. 

It is clear from \eqref{eq3.3.11}, that $S_\lambda^*(i,\ty)=c_{i,\lambda}^0 (j,\ty)$, if $i\arr\lambda j$, (recall the convention $\ty=0\ty$), and $S_\lambda^*(i,\ty)=0$ if $\lambda\not\in\Lambda_i$. Therefore $S_\lambda^*$ coincides with $V_\lambda^*$ on $\mathcal K$. 

It remains to prove that $\mathcal K$ is cyclic for the representation. We will prove, by induction, that, for all words $w\in\Omega_N^*$ of length $|w|=m$ and all vertices $i\in\mathcal V$, $(i,w)$ is in $$\mathcal S:=\Span\left\{ S_{\lambda_1\dots\lambda_p}\mathcal K : \lambda_1,\dots,\lambda_p\in\Lambda,p\geq 0\right\}.$$
For $m=0$, we have $(i,w)=(i,\ty)\in\mathcal K$, so the assertion is clear. 

Assume now, that the assertion is true for $m$. Let $kw$ an arbitrary word of length $m+1$, and $i\in\mathcal V$. Case 1: $k\geq n_i$. Take $(j,\lambda)=\varphi^{-1}(i,k)$. Then 
$$(i,kw)=(F(j,\lambda),G(j,\lambda)w)=S_\lambda(j,w)\in\mathcal S.$$

Case 2: $k<n_i$. 
$$(i,kw)=\sum_{\lambda\in\Lambda_i}\sum_{k'=0}^{n_i-1}c_{i,\lambda}^k\cj{c_{i,\lambda}^{k'}}(i,k'w)=\sum_{\lambda\in\Lambda_i} c_{i,\lambda}^kS_{\lambda}(i\cdot\lambda,w)\in\mathcal S.$$

By induction, if follows that $\mathcal S=\mathcal H$, which means that $\mathcal K$ is cyclic for the representation.
\end{proof}

\section{Intertwining operators}

The main goal in this section is to prove Theorem \ref{th4.12}, which describes the intertwining operators between the Cuntz dilations associated to two random walks. We begin with some properties of invariant sets. We assume from this point on that the graphs are finite.

\begin{proposition}\label{pr4.2}
The invariant sets have the following properties
\begin{enumerate}
	\item[(i)] Every invariant set contains a minimal invariant subset. 
	\item[(ii)] For $(i,i')\in\mathcal V\times \mathcal V'$, its orbit $\mathcal O(i,i')$ is invariant. 
	\item[(iii)] If $\mathcal M$ is a minimal invariant set and $(i,i')\in\mathcal M$, then $\mathcal O(i,i')=\mathcal M$; in other words, if $(j,j')\in\mathcal M$, then there exists a possible transition $(i,i')\arr\lambda(j,j')$.
	\item[(iv)] For every pair of vertices $(i,i')\in\mathcal V\times\mathcal V'$, there exists a minimal invariant set $\mathcal M$, such that for every pair of vertices $(j,j')\in\mathcal M$ the transition $(i,i')\arr\lambda(j,j')$ is possible. 
\end{enumerate}

\end{proposition}

\begin{proof}
(i) is obvious, just take an invariant subset of the smallest cardinality. 

(ii) If $(j,j')\in\mathcal O(i,i')$ and the transition $(j,j')\arr{\lambda_2}(k,k')$ is possible, then there is a possible transition $(i,i')\arr{\lambda_1}(j,j')$ and therefore, the transition $(i,i')\arr{\lambda_1\lambda_2}(k,k')$ is also possible, so $(k,k')\in\mathcal O(i,i')$.

(iii) Since $\mathcal M$ is invariant, it contains $\mathcal O(i,i')$. Since $\mathcal O(i,i')$ is invariant and $\mathcal M$ is minimal, $\mathcal M=\mathcal O(i,i')$. 

(iv) Consider $\mathcal O(i,i')$; it is an invariant set, therefore, it contains a minimal invariant set $\mathcal M$. Any point $(j,j')$ in $\mathcal M$ is in the orbit of $(i,i')$, so there is a possible transition from $(i,i')$ to $(j,j')$. 
\end{proof}

Heuristically, the next key lemma says that, for each pair of vertices $(i,i')$, with probability one, the random walk will reach a prescribed point in one of the minimal invariant sets. It is a generalization of the well known result that a finite irreducible random walk is recurrent. See Remark \ref{rem4.5.1} below. 
\begin{lemma}\label{lem4.3}
For each minimal invariant set $\mathcal M$, pick a point $(i_{\mathcal M},i_{\mathcal M}')$ in $\mathcal M$. For $(i,i')\in\mathcal V\times \mathcal V'$, define the following set of paths/words that do not go through any of the points $(i_{\mathcal M},i'_{\mathcal M})$:
\begin{multline}
A(i,i'):=\left\{\ \lambda=\lambda_1\dots\lambda_n : n\geq 0, (i\cdot \lambda_1\dots\lambda_k, i'\cdot\lambda_1\dots\lambda_k')\neq (i_{\mathcal M},i_{\mathcal M}')\right.\\
\left.\mbox{ for all $\mathcal M$ minimal invariant and $0\leq k\leq n$}\right\}.
\label{eq4.3.1}
\end{multline}
Define, for $(i,i')\in\mathcal V\times\mathcal V'$, $n\in\bn$,
\begin{equation}
P((i,i');n):=\sum_{|\lambda|=n, \lambda\in A(i,i')}|\alpha_{i,\lambda}||\alpha_{i',\lambda}'|.
\label{eq4.3.2}
\end{equation}
Then
\begin{equation}
\lim_{n\rightarrow\infty}P((i,i');n)=0.
\label{eq4.3.3}
\end{equation}
\end{lemma}

\begin{proof}
We prove first, that $P((i,i');n)$ is decreasing. 
$$P((i,i');n+1)=\sum_{\lambda=\lambda_1\dots\lambda_{n+1}, \lambda\in A(i,i')}|\alpha_{i,\lambda_1\dots\lambda_{n+1}}||\alpha_{i',\lambda_1\dots\lambda_{n+1}}'|$$
$$\leq \sum_{\lambda'=\lambda_1\dots\lambda_n,\lambda'\in A(i,i')}\sum_{\lambda_{n+1}\in\Lambda}|\alpha_{i,\lambda'}||\alpha_{i',\lambda'}'||\alpha_{i\cdot\lambda',\lambda_{n+1}}||\alpha_{i'\cdot\lambda',\lambda_{n+1}}'|.$$
But, by the Cauchy-Schwarz inequality,  
$$\sum_{\lambda_{n+1}\in\Lambda}|\alpha_{i\cdot\lambda',\lambda_{n+1}}||\alpha_{i'\cdot\lambda',\lambda_{n+1}}|\leq \left(\sum_{\lambda_{n+1}\in\Lambda}|\alpha_{i\cdot\lambda',\lambda_{n+1}}|^2\right)^{1/2}
\left(\sum_{\lambda_{n+1}\in\Lambda}|\alpha_{i'\cdot\lambda',\lambda_{n+1}}'|^2\right)^{1/2}=1$$
Therefore, we obtain further 
$$P((i,i');n+1)\leq \sum_{\lambda'=\lambda_1\dots\lambda_n,\lambda'\in A(i,i')}|\alpha_{i,\lambda'}||\alpha_{i',\lambda'}'|=P((i,i');n).$$

Next, we claim that, for each pair of vertices $(i,i')$, there exists $n\in\bn$ such that $P((i,i');n)<1$ and we define $n_{(i,i')}$ to be the minimal one.

Indeed, using Proposition \ref{pr4.2} (iv) and (iii), there exists some possible transition\\$(i,i')\arr{\lambda_0} (i_{\mathcal M},i'_{\mathcal M})$, for some minimal invariant set $\mathcal M$. Let $n=|\lambda_0|$. Then
$$P((i,i');n)=\sum_{|\lambda|=n,\lambda\in A(i,i')}|\alpha_{i,\lambda}||\alpha_{i',\lambda}'|\leq\sum_{|\lambda|=n}|\alpha_{i,\lambda}||\alpha_{i',\lambda}'|-|\alpha_{i,\lambda_0}||\alpha_{i',\lambda_0}'|$$
$$\leq \left(\sum_{|\lambda|=n}|\alpha_{i,\lambda}|^2\right)^{1/2}\left(\sum_{|\lambda|=n}|\alpha_{i',\lambda}'|^2\right)^{1/2}-|\alpha_{i,\lambda_0}||\alpha_{i',\lambda_0}'|=1-|\alpha_{i,\lambda_0}||\alpha_{i',\lambda_0}'|<1.$$

Now, let $L$ be the maximum of $n_{(i,i')}$ for all $(i,i')\in\mathcal V\times\mathcal V'$. Then $L\geq n_{(i,i')}$ for all $(i,i')$ so $P((i,i');L)\leq P((i,i');n_{(i,i')})<1$. Define 
$$p:=\max_{(i,i')}P((i,i');L)<1.$$

We have
$$P((i,i');(n+1)L)=\sum_{|\lambda_1|=nL,|\lambda_2|=L,\lambda_1\lambda_2\in A(i,i')}|\alpha_{i,\lambda_1\lambda_2}||\alpha_{i',\lambda_1\lambda_2}'|$$
$$=\sum_{|\lambda_1|=nL,\lambda_1\in A(i,i')}\sum_{|\lambda_2|=L, \lambda_2\in A(i\cdot\lambda_1,i'\cdot\lambda_1)}|\alpha_{i,\lambda_1}||\alpha_{i',\lambda_1}'||\alpha_{i\cdot\lambda_1,\lambda_2}||\alpha_{i\cdot\lambda_1,\lambda_2}'|$$
$$=\sum_{|\lambda_1|=nL,\lambda_1}|\alpha_{i,\lambda_1}||\alpha_{i',\lambda_1}'|P((i\cdot\lambda_1,i'\cdot\lambda_1);L)$$
$$\leq \sum_{|\lambda_1|=nL,\lambda_1\in A(i,i')}|\alpha_{i,\lambda_1}||\alpha_{i',\lambda_1}'|\cdot p=P((i,i');nL)\cdot p.$$
Therefore, $P((i,i');nL)\rightarrow 0$. Since $P((i,i');n)$ is decreasing, we get that $P((i,i');n)\rightarrow 0$. 
\end{proof}

Recall now the Definition \ref{def4.9}. Given a pair of vertices $(i,i')$, the set $F(i,i')$ consists of paths which reach one the prescribed points $(i_{\mathcal M},i'_{\mathcal M})$ for the first time. The next theorem, shows that the matrix entries of an operator $T$ in $\mathcal B(\mathcal K,\mathcal K')^{\bm \sigma}$ are completely determined by the matrix entries corresponding to $(i_{\mathcal M},i'_{\mathcal M})$.

\begin{theorem}\label{th4.4}

Let $T:\mathcal K\rightarrow\mathcal K'$ be an operator in $\mathcal B(\mathcal K,\mathcal K')^{\bm \sigma}$, so
\begin{equation}
T=\sum_{\lambda\in\Lambda}V_\lambda'TV_\lambda^*.
\label{eq4.4.1}
\end{equation}
Then 
\begin{equation}
\ip{T(i,\ty)}{(i',\ty)}=\sum_{\lambda\in F(i,i')}\alpha_{i,\lambda}\cj \alpha'_{i',\lambda}\ip{T(i\cdot\lambda,\ty)}{(i'\cdot\lambda,\ty)}.
\label{eq4.4.2}
\end{equation}
\end{theorem}

\begin{proof}
We denote by $T_{i,i'}=\ip{T(i,\ty)}{(i',\ty)}$. From \eqref{eq4.4.1}, we get 
\begin{equation}
T_{i,i'}=\sum_{\lambda\in\Lambda}\alpha_{i,\lambda}\cj\alpha'_{i',\lambda}T_{i\cdot\lambda,i'\cdot\lambda}.
\label{eq4.4.3}
\end{equation}
Iterating \eqref{eq4.4.3}, by induction, we obtain 
\begin{equation}
T_{i,i'}=\sum_{\lambda=\lambda_1\dots\lambda_n}\alpha_{i,\lambda}\cj\alpha'_{i',\lambda}T_{i\cdot\lambda,i'\cdot\lambda}.
\label{eq4.4.4}
\end{equation}
We split the sum in \eqref{eq4.4.4} into the sum over the paths $\lambda$ that go through one of the points $(i_{\mathcal M}, i_{\mathcal M}')$, and the ones that do not. We have, with the notation from Lemma \ref{lem4.3},
$$T_{i,i'}=\sum_{|\lambda|=n,\lambda\not\in A(i,i')}\alpha_{i,\lambda}\cj\alpha'_{i',\lambda}T_{i\cdot\lambda,i'\cdot\lambda}+\sum_{|\lambda|=n,\lambda\in A(i,i')}\alpha_{i,\lambda}\cj\alpha'_{i',\lambda}T_{i\cdot\lambda,i'\cdot\lambda}.$$
Since $T$ is bounded, using Lemma \ref{lem4.3} and the triangle inequality, we get that the second sum converges to 0 as $n\rightarrow \infty$. 

For the first sum, each $\lambda\not\in A(i,i')$ goes through one of the points $(i_{\mathcal M},i'_{\mathcal M})$, so we split $\lambda$ into two parts, where it reaches one of these points for the first time, $\lambda=\beta\gamma$ with $\beta\in F(i,i')$, $|\beta|\leq n$, and $|\gamma|=n-|\beta|$. Therefore we have:
$$\sum_{|\lambda|=n,\lambda\not\in A(i,i')}\alpha_{i,\lambda}\cj\alpha'_{i',\lambda}T_{i\cdot\lambda,i'\cdot\lambda}=
\sum_{|\beta|\leq n,\beta\in F(i,i')}\sum_{|\gamma|=n-|\beta|}\alpha_{i,\beta\gamma}\cj\alpha'_{i',\beta\gamma}T_{i\cdot\beta\gamma,i'\cdot\beta\gamma}$$
$$=\sum_{|\beta|\leq n,\beta\in F(i,i')}\alpha_{i,\beta}\cj\alpha_{i',\beta}'\sum_{|\gamma|=n-|\beta|}\alpha_{i\cdot\beta,\gamma}\cj\alpha'_{i'\cdot\beta,\gamma}T_{(i\cdot\beta)\cdot\gamma,(i'\cdot\beta)\cdot\gamma}
\stackrel{\eqref{eq4.4.4}}{=}\sum_{|\beta|\leq n,\beta\in F(i,i')}\alpha_{i,\beta}\cj\alpha_{i',\beta}' T_{i\cdot\beta,i'\cdot\beta}.
$$
Letting $n\rightarrow\infty$ we get \eqref{eq4.4.2}.
\end{proof}

\begin{remark}\label{rem4.5.1}
If the two graphs $(\mathcal V,\mathcal E)$ and $(\mathcal V',\mathcal E')$ are the same (including the weights), then we can always take $T=I_{\mathcal K}$ in \eqref{eq4.4.1}. Let's see what \eqref{eq4.4.2} gives us in this case. 

If $i\neq i'$, then $i\cdot\lambda\neq i'\cdot\lambda$ for any path $\lambda=\lambda_1\dots\lambda_n$, so \eqref{eq4.4.2} is trivial in this case, with both sides equal to zero. 

However, if $i=i'$, then $i\cdot\lambda=i'\cdot\lambda$ for all paths $\lambda$, and therefore \eqref{eq4.4.2} gives us the following interesting relation 
\begin{equation}
1=\sum_{\lambda\in F(i,i)}|\alpha_{i,\lambda}|^2,
\label{eq4.5.1.1}
\end{equation}
which can be interpreted as: the probability to reach one of the points $i_{\mathcal M}$ is one. This a well known fact from probability: any finite irreducible Markov chain is recurrent (see e.g. \cite[Theorem 6.6.4]{Dur10}).
\end{remark}

As a byproduct of the relation \eqref{eq4.5.1.1}, we obtain an interesting relation in the dilation space.  
\begin{theorem}\label{th4.7}
With the notations in Theorem \ref{th4.4}, assume that the two graphs are the same. Then, for all $i\in\mathcal V$, 
\begin{equation}
(i,\ty)=\sum_{\lambda\in F(i,i)}\alpha_{i,\lambda}S_\lambda(i\cdot\lambda, \ty).
\label{eq4.7.1}
\end{equation}
\end{theorem}

\begin{proof}
We begin with a well known general Lemma, see, e.g., \cite{Dav96}.

\begin{lemma}\label{lem4.8}
Let $(S_\lambda)_{\lambda\in\Lambda}$ be some representation of a Cuntz algebra on a Hilbert space $\mathcal H$ and let $h_1,h_2\in\mathcal H$ and $\lambda_1,\lambda_2$ two finite words. Then $S_{\lambda_1}h_1$ is orthogonal to $S_{\lambda_2}h_2$ unless $\lambda_1$ is a prefix of $\lambda_2$ or vice versa. 
\end{lemma}
\begin{proof}
If $\lambda_1$ and $\lambda_2$ are not prefixes, one for the other, then, there exists $1\leq k\leq \min\{|\lambda_1|,|\lambda_2|\}$ such that $\lambda_{1,j}=\lambda_{2,j}$ for $1\leq j<k$, and $\lambda_{1,k}\neq \lambda_{2,k}$. Then 
$$\ip{S_{\lambda_1}h_1}{S_{\lambda_2}h_2}=\ip{S_{\lambda_{1,1}}\dots S_{\lambda_{1,k-1}} S_{\lambda_{1,k}}\dots S_{\lambda_{1,|\lambda_1|}}h_1}{S_{\lambda_{2,1}}\dots S_{\lambda_{2,k-1}} S_{\lambda_{2,k}}\dots S_{\lambda_{2,|\lambda_2|}}h_2}$$$$=\ip{S_{\lambda_{1,k}}\dots S_{\lambda_{1,|\lambda_1|}}h_1}{S_{\lambda_{2,k}}\dots S_{\lambda_{2,|\lambda_2|}}h_2}=0.$$
\end{proof}

Using Lemma \ref{lem4.8}, we notice that words in $F(i,i)$ cannot be prefixes of each other, since $\lambda$ is in $F(i,i)$ if the path reaches one of the points $(i_{\mathcal M},i_{\mathcal M})$ for the {\it first} time. Therefore the vectors $S_\lambda(i\cdot\lambda,\ty)$, $\lambda\in F(i,i)$ form an orthonormal set. We project the vector $(i,\ty)$ onto this orthonormal set, and we compute the coefficients. 
$$\ip{(i,\ty)}{S_\lambda(i\cdot\lambda)}=\ip{S_\lambda^*(i,\ty)}{(i\cdot\lambda,\ty)}=\ip{V_\lambda^*(i,\ty)}{(i\cdot\lambda,\ty)}=\ip{\alpha_{i,\lambda}(i\cdot\lambda,\ty)}{(i\cdot\lambda,\ty)}=\alpha_{i,\lambda}.$$
This means that the right-hand side of \eqref{eq4.7.1} is the projection of $(i,\ty)$ onto the span of $(S_\lambda(i\cdot\lambda,\ty))_{\lambda\in F(i,i)}$. But the square of the norm of this projection is 
$$\sum_{\lambda\in F(i,i)}|\alpha_{i,\lambda}|^2=1=\|(i,\ty)\|^2,$$
by \eqref{eq4.5.1.1}. Thus, \eqref{eq4.7.1} follows.
\end{proof}


The next Lemma shows that, given an operator $T$ in $\mathcal B(\mathcal K,\mathcal K')^{\bm \sigma}$, the matrix entries of $T$ have to be $0$ on non-balanced minimal invariant sets, and there are important restrictions on the balanced ones. 

\begin{lemma}\label{lem4.10}
With the notations as in Theorem \ref{th4.4}, let $T_{i,i'}:=\ip{T(i,\ty)}{(i',\ty)}$, for all $(i,i')\in\mathcal V\times\mathcal V'$. We have the following two possibilities:
\begin{enumerate}
	\item [(i)] Either, $\mathcal M$ is not balanced, and then $T_{i,i'}=0$ for all $(i,i')\in\mathcal M$, or
	\item[(ii)] $\mathcal M$ is balanced and 
	\begin{equation}
	T_{i\cdot\lambda,i'\cdot\lambda}=T_{i,i'}\frac{\alpha_{i,\lambda}}{\alpha_{i',\lambda}'},
	\label{eq4.10.1}
	\end{equation}
	for all $(i,i')$ in $\mathcal M$ and all words $\lambda$ for which the transition $(i,i')\arr\lambda (j,j')$ is possible (for some $(j,j')$).
\end{enumerate}
\end{lemma}

\begin{proof}
Let's assume that $T_{i,i'}\neq 0$ for some $(i,i')\in \mathcal{M} $ and pick $(i,i')\in \mathcal{M} $ so that $|T_{i,i'}|=\max_{(j,j')\in\mathcal M}|T_{j,j'}|$. With the Schwarz inequality, we have, for all $n\geq 1$, 
$$|T_{i,i'}|=\left|\sum_{|\lambda|=n}\alpha_{i,\lambda}\cj\alpha'_{i',\lambda}T_{i\cdot\lambda,i'\cdot\lambda}\right|\leq \left(\sum_{|\lambda|=n}|\alpha_{i,\lambda}|^2|T_{i\cdot\lambda,i'\cdot\lambda}|^2\right)^{1/2}\left(\sum_{|\lambda|=n}|\alpha'_{i',\lambda}|^2\right)^{1/2}$$
$$\leq |T_{i,i'}|\cdot 1\cdot 1=|T_{i,i'}|.$$
Thus, we must have equalities in all inequalities. 

Since we have equality, there exists some constant $c=c(i,i',n)\in\bc$ such that

\begin{equation}
\alpha_{i,\lambda}T_{i\cdot\lambda,i'\cdot\lambda}=c\alpha'_{i',\lambda},
\label{eqprop}
\end{equation}
for all words $\lambda$ with $|\lambda|=n$. Thus $\alpha'_{i',\lambda}=0$ if $\alpha_{i,\lambda}=0$, and also conversely, by symmetry, which means that the transition $i\arr\lambda i\cdot\lambda$ is possible if and only if the transition $i'\arr\lambda i'\cdot\lambda$ is. Note that $c$ cannot be $0$ because of the assumptions.
Since we have equality in the second inequality, it follows that $|T_{i\cdot\lambda, i'\cdot\lambda}|=|T_{i,i'}|$ if $\alpha_{i,\lambda}\neq 0$. By Proposition \ref{pr4.2}(iii), there are possible transitions from $(i,i')$ to any other $(j,j')$ in $\mathcal M$. Thus $|T_{i,i'}|$ is constant on $\mathcal M$.

We have then 
$$T_{i,i'}=\sum_{|\lambda|=n}\alpha_{i,\lambda}\cj\alpha'_{i',\lambda}T_{i\cdot\lambda,i'\cdot\lambda}=\sum_{|\lambda|=n}c|\alpha_{i',\lambda}'|^2=c.$$
This and \eqref{eqprop} implies \eqref{eq4.10.1}.

From this, we get that $|\alpha_{i,\lambda}||T_{i\cdot\lambda, i'\cdot\lambda}|=|T_{i,i'}||\alpha'_{i',\lambda}|$ for all $|\lambda|=n$. Since $|T_{i,i'}|$ is constant on $\mathcal M$, we get that $|\alpha_{i, \lambda}|=|\alpha'_{i', \lambda}|$ when the transition $i\arr\lambda i\cdot\lambda$ is possible. If the transition is not possible, then $|\alpha_{i,\lambda}|=|\alpha'_{i',\lambda}|=0$. Thus condition (i), for $\mathcal M$ to be balanced, is satisfied. 

Now take a point $(i_{\mathcal M},i'_{\mathcal M})$ in $\mathcal M$. With \eqref{eq4.4.2}, we have, 
$$T_{i_{\mathcal M},i'_{\mathcal M}}=\sum_{\lambda\in F(i_{\mathcal M},i'_{\mathcal M})}\alpha_{i_{\mathcal M},\lambda}\cj\alpha'_{i'_{\mathcal M},\lambda} T_{i_{\mathcal M}\cdot \lambda,i'_{\mathcal M}\cdot\lambda}=
\sum_{\lambda\in F(i_{\mathcal M},i'_{\mathcal M})}\alpha_{i_{\mathcal M},\lambda}\cj\alpha'_{i'_{\mathcal M},\lambda} T_{i_{\mathcal M},i'_{\mathcal M}}.$$

Since $T_{i_{\mathcal M},i'_{\mathcal M}} \neq 0$ we get 
\begin{equation}
\sum_{\lambda\in F(i_{\mathcal M},i'_{\mathcal M})}\alpha_{i_{\mathcal M},\lambda}\cj\alpha'_{i'_{\mathcal M},\lambda}=1.
\label{eq4.10.2}
\end{equation}

Using the Schwarz inequality and \eqref{eq4.5.1.1}, we get 
$$1\leq \left(\sum_{\lambda\in F(i_{\mathcal M},i'_{\mathcal M})}|\alpha_{i_{\mathcal M},\lambda}|^2\right)^{1/2}\left(\sum_{\lambda\in F(i_{\mathcal M},i'_{\mathcal M})}|\alpha'_{i'_{\mathcal M},\lambda}|^2\right)^{1/2}=1.$$
Thus, we have equality in the Schwarz inequality, so there exists a constant $c\in\bc$ such that $\alpha_{i_{\mathcal M},\lambda}=c\alpha'_{i'_{\mathcal M},\lambda}$ for all $\lambda\in  F(i_{\mathcal M},i'_{\mathcal M})$. Using \eqref{eq4.10.2} again, we get that $c=1$. Thus $\alpha_{i_{\mathcal M},\lambda}=\alpha_{i'_{\mathcal M},\lambda}$ for all $\lambda\in F(i_{\mathcal M},i'_{\mathcal M})$. Since every loop at $(i_{\mathcal M},i'_{\mathcal M})$ is a concatenation of loops from $F(i_{\mathcal M},i'_{\mathcal M})$, and since $(i_{\mathcal M},i'_{\mathcal M})$ is arbitrary, it follows that condition (ii), for $\mathcal M$ to be balanced, is satisfied. Thus, $\mathcal M$ is balanced. 
\end{proof}

The next Lemma shows that, if we can construct an operator $T$ in $\mathcal B(\mathcal K,\mathcal K')^{\bm \sigma}$ from some arbitrary prescribed matrix entries on balanced minimal invariant sets. 

\begin{lemma}\label{lem4.11}
For each minimal invariant set $\mathcal M$, let $c_{i_{\mathcal M},i'_{\mathcal M}} =0$ if $\mathcal M$ is not balanced and let $c_{i_{\mathcal M},i'_{\mathcal M}}$ be in $\bc$, arbitrary, if $\mathcal M$ is balanced. Define the operator $T:\mathcal K\rightarrow \mathcal K'$ by
\begin{equation}
\ip{T{(i,\ty)}}{(i',\ty)}=\sum_{\lambda\in F(i,i')}\alpha_{i,\lambda}\cj\alpha_{i',\lambda} c_{i\cdot\lambda,i'\cdot\lambda}.
\label{eq4.11.1}
\end{equation}
Then the operator $T$ satisfies \eqref{eq4.4.1}.
\end{lemma}

\begin{proof}
Let $(i,i')\in\mathcal V\times\mathcal V'$. We compute 
$$S:=\sum_{\lambda_1\in\Lambda}\alpha_{i,\lambda_1}\cj\alpha_{i',\lambda_1}T_{i\cdot\lambda_1,i'\cdot\lambda_1}.$$
We split the sum in two. Consider the set $A$ of $\lambda_1\in\Lambda$ such that $(i\cdot\lambda_1,i'\cdot\lambda_1)=(i_{\mathcal M},i'_{\mathcal M})$ for some minimal invariant set $\mathcal M$. 
Then 
$$S=\sum_{\lambda_1\in A}\alpha_{i,\lambda_1}\cj\alpha_{i',\lambda_1}T_{i\cdot\lambda_1,i'\cdot\lambda_1}+\sum_{\lambda_1\not\in A}\alpha_{i,\lambda_1}\cj\alpha_{i',\lambda_1}T_{i\cdot\lambda_1,i'\cdot\lambda_1}$$
But if $\lambda_1\in A$, so $(i\cdot\lambda_1,i'\cdot\lambda_1)=(i_{\mathcal M},i'_{\mathcal M})$ for some $\mathcal M$, then, we can assume $\mathcal M$ is balanced, otherwise $T_{i\cdot\lambda_1,i'\cdot \lambda_1}=0$, and we have:
$$T_{i\cdot\lambda_1,i'\cdot\lambda_1}=T_{i_{\mathcal M},i'_{\mathcal M}}=\sum_{\lambda\in F(i_{\mathcal M},i'_{\mathcal M})}\alpha_{i_\mathcal{M},\lambda}\cj\alpha'_{i_{\mathcal{M}}',\lambda} c_{i_{\mathcal{M}}\cdot\lambda,i_{\mathcal{M}}'\cdot\lambda}=\sum_{\lambda\in F(i_{\mathcal M},i'_{\mathcal M})}\alpha_{i_{\mathcal{M}},\lambda}\cj\alpha_{i_\mathcal{M}',\lambda} c_{i_{\mathcal M},i'_{\mathcal M}}$$
$$=c_{i_{\mathcal M},i'_{\mathcal M}}\sum_{\lambda\in F(i_{\mathcal M},i'_{\mathcal M})}|\alpha_{i_{\mathcal{M}},\lambda}|^2=c_{i_{\mathcal M},i'_{\mathcal M}}.$$
Thus 
$$S=\sum_{\lambda_1\in A}\alpha_{i,\lambda_1}\cj\alpha'_{i',\lambda_1}c_{i\cdot\lambda_1,i'\cdot\lambda_1}+\sum_{\lambda_1\not\in A}\alpha_{i,\lambda_1}\cj\alpha'_{i',\lambda_1}\sum_{\lambda\in F(i\cdot\lambda_1,i'\cdot\lambda_1)}\alpha_{i\cdot\lambda_1,\lambda}\cj\alpha_{i'\cdot\lambda_1,\lambda} c_{i\cdot\lambda_1\lambda,i'\cdot\lambda_1\lambda}$$
$$=\sum_{\lambda_1\in A}\alpha_{i,\lambda_1}\cj\alpha'_{i',\lambda_1}c_{i\cdot\lambda_1,i'\cdot\lambda_1}+\sum_{\lambda_1\not\in A}\sum_{\lambda\in F(i\cdot\lambda_1,i'\cdot\lambda_1)}\alpha_{i,\lambda_1\lambda}\cj\alpha'_{i',\lambda_1\lambda} c_{i\cdot\lambda_1\lambda,i'\cdot\lambda_1\lambda}$$
$$=\sum_{\gamma\in F(i,i')}\alpha_{i,\gamma}\cj\alpha'_{i',\gamma}c_{i\cdot\gamma,i'\cdot\gamma}= T_{i,i'}.$$
This proves \eqref{eq4.4.1}.
\end{proof}

\begin{proof}[Proof of Theorem \ref{th4.12}]
Lemma \ref{lem4.11} shows that the map in \eqref{eq4.12.2} is well defined. Lemma \ref{lem4.10} shows that $T_{i,i'}$ has to be zero on non-balanced minimal sets, and then Theorem \ref{th4.4} shows that the maps are inverse to each other. It is clear that the maps are linear. 
\end{proof}

\begin{definition}\label{def5.9}
    We say that the random walk is {\it connected} if, for any pair of vertices $i,j\in\mathcal V$, there is a possible transition from $i$ to $j$. We say that the random walk is {\it separating}, if, for any pair of distinct vertices $i\neq i'$ in $\mathcal V$, there exists $n\in\bn$ such that, for any vertices $j,j'$ in $\mathcal{V}$ and for any word $\lambda$ of length $|\lambda|=n$, either the transition $i\arr\lambda j$ is not possible or the transition $i'\arr\lambda j'$ is not possible.
\end{definition}

\begin{corollary}\label{cor5.10}
If the random walk is connected and separating, then the Cuntz dilation is irreducible. 
\end{corollary}

\begin{proof}
We prove that there is only one balanced minimal invariant set. 
Let $\mathcal M$ be a balanced minimal invariant set. Suppose there is a pair $(i,i')$ in $\mathcal M$ with $i\neq i'$. Since the random walk is separating there exists $n\in\bn$ such that for every word $\lambda$ with $|\lambda|=n$, either $\alpha_{i,\lambda}=0$ or $\alpha_{i',\lambda}=0$. Since the $|\alpha_{i,\lambda}|^2$ are probabilities, there exists $\lambda$ with $|\lambda|=n$ such that $\alpha_{i,\lambda}\neq 0$. Then $\alpha_{i',\lambda}=0$, which contradicts the fact that $\mathcal M$ is balanced. 

Thus $\mathcal M$ cannot contain pairs $(i,i')$ with $i\neq i'$. Let $(i,i)\in\mathcal M$, since the random walk is connected the orbit of $(i,i)$ is $\{(j,j): j\in\mathcal V\}$. By Proposition \ref{pr4.2}, it follows that $\mathcal M=\{(j,j): j\in\mathcal V\}$. 

Hence, there is only one balanced minimal invariant set, which means, according to Theorem \ref{th4.12}, that the commutant of the Cuntz dilation is one-dimensional so the Cuntz dilation is irreducible. 
\end{proof}

\section{Examples}

\begin{example}\label{ex6.1}
Let $G$ be a finite group and let $\Lambda$ be a set of generators of $G$ in the sense that each element in $G$ is a product of elements of $\Lambda$. Let $N=|\Lambda|$. We consider the Cayley graph of the group: the vertices are the elements of $G$, and, for each $g\in G$, there is an edge from $g$ to $\lambda g$, for all $\lambda\in\Lambda$. We will take the numbers $\alpha_{g,\lambda}=\frac{1}{\sqrt N}$, for all $g\in G$, $\lambda\in\Lambda$. This corresponds to equal probability of transition, for all $\lambda\in\Lambda$. 

Then, the Hilbert space $\mathcal K=l^2(G)$.  The canonical vectors are $\vec g=\delta_g$, $g\in G$. For the row-coisometry, we have $V_\lambda^*\delta_g=\frac{1}{\sqrt N}\delta_{\lambda g}$, which means that $V_\lambda^*=\frac{1}{\sqrt{N}}L(\lambda)$, $\lambda\in\Lambda$, where $(L(g))_{g\in G}$ is the left regular representation of $G$. 

Next we compute the minimal invariant sets. Let $\mathcal M$ be a minimal invariant set. Take $(g_1,g_2)\in \mathcal M$. Then, $\mathcal M=\mathcal O(g_1,g_2)$, by Proposition \ref{pr4.2}(iii). So, for $(x ,y)\in \mathcal M$, there exists a word $\lambda=\lambda_1\dots\lambda_n$ such that $g_1\arr\lambda x $ and $g_2\arr\lambda y$, so $x =\lambda_n\dots\lambda_1 g_1$ and $y=\lambda_n\dots\lambda_1 g_2$. Then $x ^{-1}y=g_1^{-1}g_2$. Conversely, if $x ^{-1}y=g_1^{-1}g_2$, then $yg_2^{-1}=x g_1^{-1}=\lambda_n\dots\lambda_1$ for some elements $\lambda_1,\dots,\lambda_n$ in $\Lambda$. Then $x =\lambda_n\dots\lambda_1g_1$ and $y=\lambda_n\dots\lambda_1g_2$, which means that the transitions $g_1\arr\lambda x $, $g_2\arr\lambda y$ are possible with the word $\lambda_1\dots\lambda_n$. 

Thus each minimal invariant set is of the form 
$$\mathcal M=\{(x,y)\in G\times G : x^{-1}y=g\},$$
for some $g\in G$. Note that the minimal invariant sets form a partition of $G\times G$, corresponding to the right-cosets of the diagonal subgroup in $G\times G$. 

All minimal invariant sets are balanced, because the transition probabilities are equal. 

Next, we compute the elements in $\mathcal B(\mathcal K)^{\bm \sigma}$. Let $T$ in $\mathcal B(\mathcal K)^{\bm \sigma}$. By Lemma \ref{lem4.10}(ii), $T_{x,y}$ is constant on each minimal invariant set. Define, for all $g\in G$, 
$$T(g)_{x,y}= \left\{\begin{array}{cc}
1,&\mbox{ if }x^{-1}y=g,\\
0,&\mbox{ otherwise.}
\end{array},
\right.\quad(x,y\in G).$$
Then $T(g)$ is in $\mathcal B(\mathcal K)^{\bm\sigma}$ and every operator in $\mathcal B(\mathcal K)^{\bm\sigma}$ is a linear combination $\sum_{g\in G}a_g T(g)$. 

Note that $T(g)\delta_a=\delta_{ag^{-1}}$, $a,g\in G$. Therefore $T(g)=R(g)$, where $(R(g))_{g\in G}$ is the right regular representation of $G$. Thus $\mathcal B(\mathcal K)^{\bm\sigma}$ is the linear span of the right regular representation. 

Moving on to the Cuntz dilation, for each $g\in G$, let $A(g)$ be the operator in the commutant of the Cuntz dilation that corresponds to $T(g)$, by the map $A\mapsto P_{\mathcal K}A P_{\mathcal K}=T$. The map is a linear $\ast$ isomorphism. We check that it preserves the product too. 

By Proposition \ref{pr2.4}, for $g_1,g_2\in G$, the operator $A(g_1)A(g_2)$ corresponds to the operator $T(g_1)\ast T(g_2)$ which is the SOT-limit of 
$$\sum_{|\lambda|=n}V_\lambda T(g_1)T(g_2)V_\lambda^*.$$
But $T(g_1)T(g_2)=R(g_1)R(g_2)=R(g_1g_2)=T(g_1g_2)$ which is also in $\mathcal B(\mathcal K)^{\bm\sigma}$. Therefore, 
$$\sum_{|\lambda|=n}V_\lambda T(g_1)T(g_2)V_\lambda^*=T(g_1g_2),$$
so $T(g_1)\ast T(g_2)=T(g_1)T(g_2)=T(g_1g_2)$, which corresponds to $A(g_1g_2)$. So $A(g_1)A(g_2)=A(g_1g_2)$, which means that the isomorphism preserves the product too.

\end{example}


\tikzset{
    to/.style={
        ->,
        thick,
        shorten <= 1pt,
        shorten >= 1pt,},
    from/.style={
        <-,
        thick,
        shorten <= 1pt,
        shorten >= 1pt,}
}

  \begin{example}\label{ex6.2}
    We consider the groups $\Z / M\Z$ for $M \in \N$. As with the Example \ref{ex6.1}, we will take the generators to be $\Lambda = \left\{ +1, -1 \right\} $, with the random walk having equal probabilities and positive $\alpha_{i, \lambda}$. Let $\mathcal{V}$ be the following graph:
    
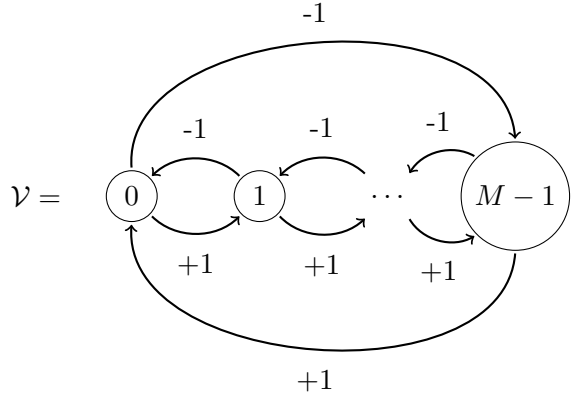
\begin{center}
        \begin{tikzpicture}[scale=0.85]
            \node[shape=circle] (V) at (-1.5,0) {$\mathcal{V} = $};
            \node[shape=circle, draw=black] (0) at (0,0) {$0$};
            \node[shape=circle, draw=black] (1) at (2,0) {$1$};
            \node[shape=circle] (ldots) at (4,0) {$\ldots$};
            \node[shape=circle, draw=black] (M-1) at (6,0) {$M-1$};

            \path (0) edge[to, bend right=45] node[label=below:{+1}] {} (1);
            \path (1) edge[to, bend right=45] node[label=below:{+1}] {} (ldots);
            \path (ldots) edge[to, bend right=45] node[label=below:{+1}] {} (M-1);
            \path (M-1) edge[to, bend left=90] node[label=below:{+1}] {} (0);

            \path (0) edge[from, bend left=45] node[label=above:{-1}] {} (1);
            \path (1) edge[from, bend left=45] node[label=above:{-1}] {} (ldots);
            \path (ldots) edge[from, bend left=45] node[label=above:{-1}] {} (M-1);
            \path (M-1) edge[from, bend right=90] node[label=above:{-1}] {} (0);
        \end{tikzpicture}
    \end{center}
    Via the Example \ref{ex6.1}, we already know that the minimal invariant sets of $\mathcal{V}\times  \mathcal{V}$ are in one-to-one correspondence with the cosets of the diagonal subgroup 
    \[
    H = \left\{ (i,i) : \, i \in \Z / M\Z \right\} \trianglelefteq \Z / M\Z \times \Z / M\Z
    .\] 
    We also want to understand the intertwiners between the graphs of $\Z / M\Z$ and $\Z / N\Z$. To this end, let $\mathcal{V}'$ be the labeled random walk associated to $\Z / N\Z$, with equal probabilities of transition, and positive $\alpha_{i',\lambda}'$. We see that:
    \begin{itemize}
            \item[i.] For each $(i,i') \in  \mathcal{V} \times \mathcal{V}'$, $\mathcal{O} (i,i')$ is a balanced, minimal invariant set.
            \item[ii.] The dimension of the intertwiners is exactly $\operatorname{gcd}(M,N)$.
    \end{itemize}
    \begin{proof}
        (i) A quick calculation shows that $\mathcal{O} (i,i') = \left\{ (i +_{M} k, i' +_{N} k) : k \in \N \right\}$, which is by definition invariant. It is also minimal since for every two elements in the orbit $(l,l'), (j,j') \in \mathcal{O} (i,i')$, the transition $\left( l,l' \right)\to \left( j,j' \right)$ is possible. Further, it follows directly from the assumption that $\alpha_{i, \lambda} = \alpha_{i', \lambda}'$ is constant for all $i, i', \lambda$, that $\mathcal{O}(i,i')$ is also balanced.

        (ii)  Recall that every minimal invariant set can be written as the orbit of some element. Now, consider the product group $\Z / M\Z \times \Z / N\Z$, and the cyclic subgroup $H = \langle (1,1) \rangle$. There is a one-to-one correspondence between elements of the quotient group $(\Z / M\Z \times \Z / N\Z) / H$ and minimal invariant sets (given by $\phi\left( (i,i')H \right) = \mathcal{O} (i,i')$). The cardinality of $H$ is $\operatorname{lcm} \left( M,N \right)$. Thus, the cardinality of the quotient group, and thus the number of minimal invariant sets, is $\frac{MN}{\operatorname{lcm} \left( M,N \right) } = \gcd(M,N)$.
    \end{proof}
		\end{example}

        \begin{example}\label{ex6.3}
We illustrate our theory here with another example. Consider the following graphs:
    
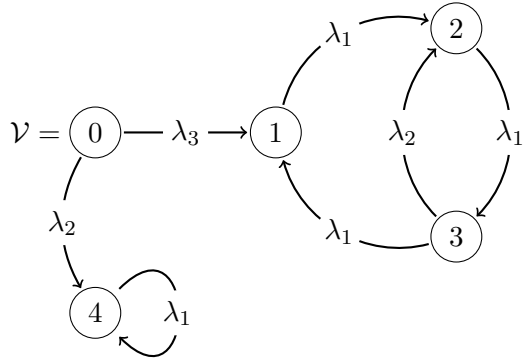
\begin{center}
        \begin{tikzpicture}[scale=0.8]
            \node[shape=circle] (mathcalV) at (-6,0) {$\mathcal{V}=$};
            \node[shape=circle, draw=black] (0) at (-5,0) {$0$};
            \node[shape=circle, draw=black] (1) at (180:2) {$1$};
            \node[shape=circle, draw=black] (2) at (060:2) {$2$};
            \node[shape=circle, draw=black] (3) at (300:2) {$3$};
            \node[shape=circle, draw=black] (4) at (-5,-3) {$4$};

            \path (0) edge[to] node[shape=circle, fill=white, minimum size=6mm, label=center:{$\lambda_{3}$}] {} (1);
            \path (1) edge[to, bend left=45] node[shape=circle, fill=white, minimum size=6mm, label=center:{$\lambda_{1}$}] {} (2);
            \path (2) edge[to, bend left=45] node[shape=circle, fill=white, minimum size=6mm, label=center:{$\lambda_{1}$}] {} (3);
            \path (3) edge[to, bend left=45] node[shape=circle, fill=white, minimum size=6mm, label=center:{$\lambda_{2}$}] {} (2);
            \path (3) edge[to, bend left=45] node[shape=circle, fill=white, minimum size=6mm, label=center:{$\lambda_{1}$}] {} (1);
            \path (0) edge[to, bend right=30] node[shape=circle, fill=white, minimum size=6mm, label=center:{$\lambda_{2}$}] {} (4);
            \path (4) edge[->, in=-45, out=45, loop, min distance = 2cm, thick, shorten <= 3pt, shorten >= 3pt] node[shape=circle, fill=white, minimum size=6mm, label=center:{$\lambda_{1}$}] {} (4);
        \end{tikzpicture}\\
        
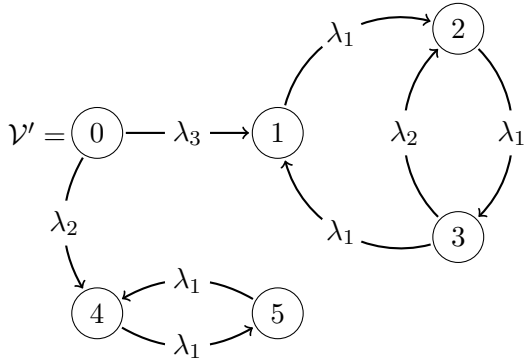
\begin{tikzpicture}[scale=0.8]
            \node[shape=circle] (mathcalV) at (-6,0) {$\mathcal{V}'=$};
            \node[shape=circle, draw=black] (0) at (-5,0) {$0$};
            \node[shape=circle, draw=black] (1) at (180:2) {$1$};
            \node[shape=circle, draw=black] (2) at (060:2) {$2$};
            \node[shape=circle, draw=black] (3) at (300:2) {$3$};
            \node[shape=circle, draw=black] (4) at (-5,-3) {$4$};
            \node[shape=circle, draw=black] (5) at (-2,-3) {$5$};

            \path (0) edge[to] node[shape=circle, fill=white, minimum size=6mm, label=center:{$\lambda_{3}$}] {} (1);
            \path (1) edge[to, bend left=45] node[shape=circle, fill=white, minimum size=6mm, label=center:{$\lambda_{1}$}] {} (2);
            \path (2) edge[to, bend left=45] node[shape=circle, fill=white, minimum size=6mm, label=center:{$\lambda_{1}$}] {} (3);
            \path (3) edge[to, bend left=45] node[shape=circle, fill=white, minimum size=6mm, label=center:{$\lambda_{2}$}] {} (2);
            \path (3) edge[to, bend left=45] node[shape=circle, fill=white, minimum size=6mm, label=center:{$\lambda_{1}$}] {} (1);
            \path (0) edge[to, bend right=30] node[shape=circle, fill=white, minimum size=6mm, label=center:{$\lambda_{2}$}] {} (4);
            \path (5) edge[to, bend right=30] node[shape=circle, fill=white, minimum size=6mm, label=center:{$\lambda_{1}$}] {} (4);
            \path (4) edge[to, bend right=30] node[shape=circle, fill=white, minimum size=6mm, label=center:{$\lambda_{1}$}] {} (5);
        \end{tikzpicture}
    \end{center}

    We will again consider the case where all $\alpha_{i,\lambda}, \alpha_{i', \lambda}'$ are positive, and the probabilities of transition are split evenly between all possible transitions.
    \begin{align*}
        \alpha_{0,\lambda_{1}} = 0\quad &
        \alpha_{0,\lambda_{2}} = \alpha_{0,\lambda_{3}} = \frac{1}{\sqrt{2}}\\
        \alpha_{1,\lambda_{2}} = \alpha_{1,\lambda_{3}} = 0\quad &
        \alpha_{1, \lambda_{1}} = 1\\
        \alpha_{2,\lambda_{2}} = \alpha_{2,\lambda_{3}} = 0\quad &
        \alpha_{2, \lambda_{1}} = 1\\
        \alpha_{3,\lambda_{3}} = 0\quad &
        \alpha_{3,\lambda_{1}} = \alpha_{3,\lambda_{2}} = \frac{1}{\sqrt{2}}\\
        \alpha_{4,\lambda_{2}} = \alpha_{4,\lambda_{3}} = 0\quad &
        \alpha_{4, \lambda_{1}} = 1\\
        \alpha'_{0,\lambda_{1}} = 0\quad &
        \alpha'_{0,\lambda_{2}} = \alpha'_{0,\lambda_{3}} = \frac{1}{\sqrt{2}}\\
        \alpha'_{1,\lambda_{2}} = \alpha'_{1,\lambda_{3}} = 0\quad &
        \alpha'_{1, \lambda_{1}} = 1\\
        \alpha'_{2,\lambda_{2}} = \alpha'_{2,\lambda_{3}} = 0\quad &
        \alpha'_{2, \lambda_{1}} = 1\\
        \alpha'_{3,\lambda_{3}} = 0\quad &
        \alpha'_{3,\lambda_{1}} = \alpha'_{3,\lambda_{2}} = \frac{1}{\sqrt{2}}\\
        \alpha'_{4,\lambda_{2}} = \alpha'_{4,\lambda_{3}} = 0\quad &
        \alpha'_{4, \lambda_{1}} = 1\\
        \alpha'_{5,\lambda_{2}} = \alpha'_{5,\lambda_{3}} = 0\quad &
        \alpha'_{5, \lambda_{1}} = 1\\
    \end{align*}
    To analyze the reducibility of the Cuntz dilation of $\mathcal{V}$, we consider the graph of $\mathcal{V}\times \mathcal{V}$. 
    We have the following orbits on $\mathcal{V}\times \mathcal{V}$:
    \begin{align*}
        &\mathcal{O}\left( 0,0 \right) = \left\{ (i,i) : \, i \in \mathcal{V} \right\}\\
        &\mathcal{O}\left( 1,1 \right) = \left\{ (1,1), (2,2), (3,3) \right\} = \mathcal{O}\left( 2,2 \right) = \mathcal{O}\left( 3,3 \right)\\
        &\mathcal{O}\left( 4,4 \right) = \left\{ (4,4) \right\}\\
        &\mathcal{O}\left( 0,1 \right) = \left\{ (0,1) \right\}\\
        &\mathcal{O}\left( 0,2 \right) = \left\{ (0,2) \right\}\\
        &\mathcal{O}\left( 0,3 \right) = \left\{ (0,3), (4,2), (4,3), (4,1) \right\}\\
        &\mathcal{O}\left( 0,4 \right) = \left\{ (0,4) \right\}\\
        &\mathcal{O}\left( 1,0 \right) = \left\{ (1,0) \right\}\\
        &\mathcal{O}\left( 1,2 \right) = \left\{ (1,2), (2,3),(3,1) \right\} = \mathcal{O}\left( 2,3 \right) = \mathcal{O}\left( 3,1 \right)\\
        &\mathcal{O}\left( 1,3 \right) = \left\{ (1,3), (2,1), (3,2) \right\} = \mathcal{O}\left( 2,1 \right) = \mathcal{O}\left( 3,2 \right)\\
        &\mathcal{O}\left( 1,4 \right) = \left\{ (1,4), (2,4), (3,4) \right\} = \mathcal{O}\left( 2,4 \right) = \mathcal{O}\left( 3,4 \right)\\
        &\mathcal{O}\left( 2,0 \right) = \left\{ (2,0) \right\}\\
        &\mathcal{O}\left( 3,0 \right) = \left\{ (3,0), (2,4), (3,4), (1,4) \right\}\\
        &\mathcal{O}\left( 4,0 \right) = \left\{ (4,0) \right\}\\
        &\mathcal{O}\left( 4,1 \right) = \left\{ (4,1), (4,2), (4,3) \right\} = \mathcal{O}\left( 4,2 \right) = \mathcal{O}\left( 4,3 \right).
    \end{align*}
    Which gives us the following minimal invariant sets:
    \begin{align*}
        &\mathcal{O}\left( 1,1 \right)
        &\mathcal{O}\left( 4,4 \right)\\
        &\mathcal{O}\left( 0,1 \right)
        &\mathcal{O}\left( 0,2 \right)\\
        &\mathcal{O}\left( 0,4 \right)
        &\mathcal{O}\left( 1,0 \right)\\
        &\mathcal{O}\left( 1,2 \right)
        &\mathcal{O}\left( 1,3 \right)\\
        &\mathcal{O}\left( 1,4 \right)
        &\mathcal{O}\left( 2,0 \right)\\
        &\mathcal{O}\left( 4,0 \right)
        &\mathcal{O}\left( 4,1 \right).
    \end{align*}
    Of which only the following are balanced:
    \begin{align*}
        &\mathcal{O}\left( 1,1 \right)
        &\mathcal{O}\left( 4,4 \right),
    \end{align*}
    since all the other minimal invariants do not meet condition (i) of Definition \ref{def4.9}.
    So the commutant of the Cuntz dilation of $\mathcal{V}$ is two-dimensional. Similarly, we can compute the dimension of the commutant of the Cuntz dilation of $\mathcal{V}'$. We see that the orbits are:
    \begin{align*}
        \mathcal{O}\left( 0,0 \right) &= \left\{ (i,i) : \, i \in \mathcal{V} \right\}\\
        \mathcal{O}\left( 1,1 \right) &= \left\{ (1,1), (2,2), (3,3) \right\} = \mathcal{O}\left( 2,2 \right) = \mathcal{O}\left( 3,3 \right)\\
        \mathcal{O}\left( 4,4 \right) &= \left\{ (4,4), (5,5) \right\} = \mathcal{O}\left( 5,5 \right)\\
        \mathcal{O}\left( 0,1 \right) &= \left\{ (0,1) \right\}\\
        \mathcal{O}\left( 0,2 \right) &= \left\{ (0,2) \right\}\\
        \mathcal{O}\left( 0,3 \right) &= \left\{ (0,3), (4,2), (5,3), (4,1), (5,2), (4,3), (5,1) \right\}\\
        \mathcal{O}\left( 0,4 \right) &= \left\{ (0,4) \right\}\\
        \mathcal{O}\left( 0,5 \right) &= \left\{ (0,5) \right\}\\
        \mathcal{O}\left( 1,0 \right) &= \left\{ (1,0) \right\}\\
        \mathcal{O}\left( 1,2 \right) &= \left\{ (1,2), (2,3),(3,1) \right\} = \mathcal{O}\left( 2,3 \right) = \mathcal{O}\left( 3,1 \right)\\
        \mathcal{O}\left( 1,3 \right) &= \left\{ (1,3), (2,1), (3,2) \right\} = \mathcal{O}\left( 2,1 \right) = \mathcal{O}\left( 3,2 \right)\\
        \mathcal{O}\left( 1,4 \right) &= \left\{ (1,4), (2,5), (3,4), (1,5), (2,4), (3,5) \right\} =\\
               &=\mathcal{O}\left( 2,5 \right) = \mathcal{O}\left( 3,4 \right) = \mathcal{O}\left( 1,5 \right) = \mathcal{O}\left( 2,4 \right) = \mathcal{O}\left( 3,5 \right)\\
        \mathcal{O}\left( 2,0 \right) &= \left\{ (2,0) \right\}\\
        \mathcal{O}\left( 3,0 \right) &= \left\{ (3,0), (2,4), (3,5), (1,4), (2,5), (3,4), (1,5) \right\}\\
        \mathcal{O}\left( 4,0 \right) &= \left\{ (4,0) \right\}\\
        \mathcal{O}\left( 4,1 \right) &= \left\{ (4,1), (5,2), (4,3), (5,1), (4,2), (5,3) \right\}\\
               &=\mathcal{O}\left( 5,2 \right) = \mathcal{O}\left( 4,3 \right) = \mathcal{O}\left( 5,1 \right) = \mathcal{O}\left( 4,2 \right) = \mathcal{O}\left( 5,3 \right)\\
        \mathcal{O}\left( 4,5 \right) &= \left\{ (4,5), (5,4) \right\} = \mathcal{O}\left( 5,4 \right).
    \end{align*}
    These yield the following minimal invariant sets:
    \begin{align*}
        \mathcal{O}\left( 1,1 \right) &= \left\{ (1,1), (2,2), (3,3) \right\} = \mathcal{O}\left( 2,2 \right) = \mathcal{O}\left( 3,3 \right)\\
        \mathcal{O}\left( 4,4 \right) &= \left\{ (4,4), (5,5) \right\} = \mathcal{O}\left( 5,5 \right)\\
        \mathcal{O}\left( 0,1 \right) &= \left\{ (0,1) \right\}\\
        \mathcal{O}\left( 0,2 \right) &= \left\{ (0,2) \right\}\\
        \mathcal{O}\left( 0,4 \right) &= \left\{ (0,4) \right\}\\
        \mathcal{O}\left( 0,5 \right) &= \left\{ (0,5) \right\}\\
        \mathcal{O}\left( 1,0 \right) &= \left\{ (1,0) \right\}\\
        \mathcal{O}\left( 1,2 \right) &= \left\{ (1,2), (2,3),(3,1) \right\} = \mathcal{O}\left( 2,3 \right) = \mathcal{O}\left( 3,1 \right)\\
        \mathcal{O}\left( 1,3 \right) &= \left\{ (1,3), (2,1), (3,2) \right\} = \mathcal{O}\left( 2,1 \right) = \mathcal{O}\left( 3,2 \right)\\
        \mathcal{O}\left( 1,4 \right) &= \left\{ (1,4), (2,5), (3,4), (1,5), (2,4), (3,5) \right\} =\\
                &=\mathcal{O}\left( 2,5 \right) = \mathcal{O}\left( 3,4 \right) = \mathcal{O}\left( 1,5 \right) = \mathcal{O}\left( 2,4 \right) = \mathcal{O}\left( 3,5 \right)\\
        \mathcal{O}\left( 2,0 \right) &= \left\{ (2,0) \right\}\\
        \mathcal{O}\left( 4,0 \right) &= \left\{ (4,0) \right\}\\
        \mathcal{O}\left( 4,1 \right) &= \left\{ (4,1), (5,2), (4,3), (5,1), (4,2), (5,3) \right\}\\
                &=\mathcal{O}\left( 5,2 \right) = \mathcal{O}\left( 4,3 \right) = \mathcal{O}\left( 5,1 \right) = \mathcal{O}\left( 4,2 \right) = \mathcal{O}\left( 5,3 \right)\\
        \mathcal{O}\left( 4,5 \right) &= \left\{ (4,5), (5,4) \right\} = \mathcal{O}\left( 5,4 \right).
    \end{align*}
    Of which, the following are balanced:
    \begin{align*}
        &\mathcal{O}\left( 1,1 \right) = \left\{ (1,1), (2,2), (3,3) \right\} = \mathcal{O}\left( 2,2 \right) = \mathcal{O}\left( 3,3 \right)\\
        &\mathcal{O}\left( 4,4 \right) = \left\{ (4,4), (5,5) \right\} = \mathcal{O}\left( 5,5 \right)\\
        &\mathcal{O}\left( 4,5 \right) = \left\{ (4,5), (5,4) \right\} = \mathcal{O}\left( 5,4 \right).
    \end{align*}
    So we have that the Cuntz dilation of $\mathcal{V}'$ has dimension $3$. Further, for any $T \in \mathcal B^{\bm \sigma}$, $T$ has the form:
    \[
        T = \begin{bmatrix} 
            \frac{a+b}{2} & 0 & 0 & 0 & 0 & 0\\
            0 & a & 0 & 0 & 0 & 0\\
            0 & 0 & a & 0 & 0 & 0\\
            0 & 0 & 0 & a & 0 & 0\\
            0 & 0 & 0 & 0 & b & c\\
            0 & 0 & 0 & 0 & c & b\\
        \end{bmatrix}, \quad a,b,c \in \C
    ,\]
    where the entry $T_{0,0} = \sum_{\lambda \in \Lambda} \alpha_{0,\lambda}, \overline{\alpha_{0,\lambda}'}T_{0\cdot\lambda, 0\cdot\lambda} = \frac{1}{2}(T_{1,1} + T_{4,4})$.

    Now, we look at the intertwiners between the Cuntz dilations of $\mathcal{V}$ and $\mathcal{V}'$. By the same process we first identify the balanced, minimal, invariant sets of $\mathcal{V}\times \mathcal{V}'$ as:
    \begin{align*}
        &\mathcal{O}\left( 1,1 \right) = \left\{ (1,1), (2,2), (3,3) \right\} = \mathcal{O}\left( 2,2 \right) = \mathcal{O}\left( 3,3 \right)\\
        &\mathcal{O}\left( 4,4 \right) = \left\{ (4,4), (4,5) \right\} = \mathcal{O}\left( 4,5 \right).
    \end{align*}
    So then the space of intertwiners between the Cuntz dilations of $\mathcal{V}$ and $\mathcal{V}'$ has dimension $2$, and any such intertwiner has the form:
    \[
        T = \begin{bmatrix} 
            \frac{a+b}{2} & 0 & 0 & 0 & 0\\
            0 & a & 0 & 0 & 0\\
            0 & 0 & a & 0 & 0\\
            0 & 0 & 0 & a & 0\\
            0 & 0 & 0 & 0 & b\\
            0 & 0 & 0 & 0 & b\\
        \end{bmatrix}, \quad a,b \in \C
    ,\]
    where the entry $T_{0,0} = \sum_{\lambda \in \Lambda} \alpha_{0,\lambda}, \overline{\alpha_{0,\lambda}'}T_{0\cdot\lambda, 0\cdot\lambda} = \frac{1}{2}(T_{1,1} + T_{4,4})$.

\end{example}

\begin{example}\label{ex6.4}
        This example will show that the phase of the numbers $\alpha_{i,\lambda}$ matters for the reducibility of the Cuntz dilation. 

    We consider the following graph:
    
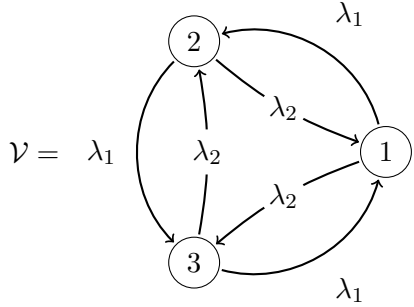
\begin{center}
        \begin{tikzpicture}[scale=0.85]
            \node[shape=circle] (V) at (180:3.5) {$\mathcal{V}=$};
            \node[shape=circle, draw=black] (1) at (000:2) {$1$};
            \node[shape=circle, draw=black] (2) at (120:2) {$2$};
            \node[shape=circle, draw=black] (3) at (240:2) {$3$};

            \path (1) edge[to,  bend right=45] node[label=060:{$\lambda_{1}$}] {} (2);
            \path (2) edge[to,  bend right=45] node[label=180:{$\lambda_{1}$}] {} (3);
            \path (3) edge[to,  bend right=45] node[label=300:{$\lambda_{1}$}] {} (1);
            \path (1) edge[from, bend left=10] node[shape=circle, fill=white, minimum size=6mm, label=center:{$\lambda_{2}$}] {} (2);
            \path (2) edge[from, bend left=10] node[shape=circle, fill=white, minimum size=6mm, label=center:{$\lambda_{2}$}] {} (3);
            \path (3) edge[from, bend left=10] node[shape=circle, fill=white, minimum size=6mm, label=center:{$\lambda_{2}$}] {} (1);
        \end{tikzpicture}
    \end{center}
    with the $\alpha_{i,\lambda}$ defined to be:
    \begin{align*}
        \alpha_{1,\lambda_{1}} = \frac{1}{\sqrt{2} }, \quad\quad &\alpha_{1, \lambda_{2}} = \frac{1}{\sqrt{2} }\\
        \alpha_{2,\lambda_{1}} = \frac{i}{\sqrt{2} }, \quad\quad &\alpha_{2, \lambda_{2}} = \frac{1}{\sqrt{2} }\\
        \alpha_{3,\lambda_{1}} = -\frac{1}{\sqrt{2}}, \quad\quad &\alpha_{2, \lambda_{2}} = \frac{1}{\sqrt{2} }.
    \end{align*}

    We see that the minimal invariant sets of $\mathcal{V}\times \mathcal{V}$ are:
    \begin{align*}
        \mathcal{M}_{1} = \left\{ (1,1), (2,2), (3,3) \right\}\\
        \mathcal{M}_{2} = \left\{ (1,2), (2,3), (3,1) \right\}\\
        \mathcal{M}_{3} = \left\{ (1,3), (2,1), (3,2) \right\} .
    \end{align*}
    So then we look to see if these minimal invariants are balanced. For each minimal invariant set $\mathcal{M}$, take $(i_{\mathcal{M}}, i'_{\mathcal{M}})$ to be the first one in the above set. 

    For each minimal invariant, we clearly have condition (i) since $\left| \alpha_{i, \lambda} \right|  = \left| \alpha_{i', \lambda} \right|$ for all $\lambda\in \Lambda$ and all $i\in \mathcal{V}$. However, for the minimal invariants $\mathcal{M}_{2}$ and $\mathcal{M}_{3}$, condition (ii) does not hold since for the loops $\left( 1,2 \right)\arr{\lambda_1\lambda_2}\left( 1,2 \right)$ and $\left( 1,3 \right)\arr{\lambda_1\lambda_2}\left( 1,3 \right)$:
    \begin{align*}
        &\alpha_{1,\lambda_{1}}\alpha_{1 \cdot \lambda_{1}, \lambda_{2}} = \frac{1}{2}\\
        &\alpha_{2,\lambda_{1}}\alpha_{2 \cdot \lambda_{1}, \lambda_{2}} = \frac{i}{2}\\
        &\alpha_{3,\lambda_{1}}\alpha_{3 \cdot \lambda_{1}, \lambda_{2}} = \frac{-1}{2}.
    \end{align*}

   Therefore, for some arbitrary $T \in \mathcal B^{\bm \sigma}$, we have that
   \begin{align*}
        T_{i,i'} = 0, \quad (i,i') \in \mathcal{M}_{2} \cup \mathcal{M}_{3}\\
        T_{i,i} = T_{j,j}, \quad i,j = 1,2,3.
   \end{align*}
   So $T = cI$ for some $c \in \C$, and the Cuntz dilation of $\mathcal{V}$ is irreducible. 
   
   In contrast, we know from Example \ref{ex6.2}, that if all the $\alpha_{i, \lambda}$ are equal (i.e. we consider the graph of $\mathbb{Z} / 3\mathbb{Z}$ with equal probabilities of transition and real, positive $\alpha_{i, \lambda}$), then the Cuntz dilation is reducible with $\dim B^{\sigma} = 3$. Thus the choice of phase for the $\alpha_{i,\lambda}$ does matter.
\end{example}

\begin{acknowledgements}
We would like to thank professor Deguang Han for very helpful conversations. We would like thank the referee for the suggestions and for pointing out some important references. 
\end{acknowledgements}

\bibliographystyle{alpha}	
\bibliography{eframes}
\end{document}